\documentclass[a4paper,oneside,11pt,reqno]{amsart}
\usepackage[margin=2.8cm, marginpar=2.3cm]{geometry}
\usepackage[utf8]{inputenc}
\usepackage{stmaryrd}
\usepackage{bbm}
\usepackage[dvipsnames]{xcolor}
\usepackage{mathrsfs}
\usepackage{enumerate}
\usepackage{centernot}
\usepackage{latexsym,amsxtra}
\usepackage{dsfont}
\usepackage{slashed}
\usepackage[all]{xy}
\usepackage{amscd,graphics}
\usepackage{amsmath,amsfonts,amsthm,amssymb}
\usepackage{latexsym,amsmath}
\usepackage{graphicx,psfrag}
\usepackage{mathabx}
\usepackage{fancyhdr}
\usepackage[color=blue!20!white,textsize=tiny]{todonotes}
\usepackage{xcolor}

\usepackage{cleveref}
\usepackage{comment} 
\usepackage{braket}
\usepackage{caption}
\usepackage{enumitem}

\theoremstyle{plain}
\newtheorem{thm}{Theorem}[section]
\crefname{thm}{Theorem}{Theorems}
\Crefname{thm}{Theorem}{Theorems}
\newtheorem{pro}[thm]{Proposition}
\crefname{pro}{Proposition}{Propositions}
\Crefname{pro}{Proposition}{Propositions}
\newtheorem{lem}[thm]{Lemma}
\crefname{lem}{Lemma}{Lemmas}
\Crefname{lem}{Lemma}{Lemmas}
\newtheorem{cor}[thm]{Corollary}
\crefname{cor}{Corollary}{Corollaries}
\Crefname{cor}{Corollary}{Corollaries}

\crefname{conj}{Conjecture}{Conjectures}
\Crefname{conj}{Conjecture}{Conjectures}

\crefname{cons}{Construction}{Constructions}
\Crefname{cons}{Construction}{Constructions}

\crefname{claim}{Claim}{Claims}
\Crefname{claim}{Claim}{Claims}

\crefname{property}{Property}{Properties}
\Crefname{property}{Property}{Properties}

\crefname{problem}{Problem}{Problems}
\Crefname{problem}{Problem}{Problems}

\theoremstyle{definition}
\newtheorem{defi}[thm]{Definition}
\crefname{defi}{Definition}{Definitions}
\Crefname{defi}{Definition}{Definitions}

\crefname{nota}{Notation}{Notations}
\Crefname{nota}{Notation}{Notations}

\crefname{convention}{Convention}{Conventions}
\Crefname{convention}{Convention}{Conventions}

\crefname{cond}{Condition}{Conditions}
\Crefname{cond}{Condition}{Conditions}
\newtheorem{assum}[thm]{Assumption}
\crefname{assum}{Assumption}{Assumptions}
\Crefname{assum}{Assumption}{Assumptions}

\theoremstyle{remark}
\newtheorem{rmk}[thm]{Remark}
\crefname{rmk}{Remark}{Remarks}
\Crefname{rmk}{Remark}{Remarks}
\newtheorem{ex}[thm]{Example}
\crefname{ex}{Example}{Examples}
\Crefname{ex}{Example}{Examples}

\newtheorem{ques}[thm]{Question}
\crefname{ques}{Question}{Questions}
\Crefname{ques}{Question}{Questions}

\crefname{section}{Section}{Sections}
\Crefname{section}{Section}{Sections}
\crefname{subsection}{Subsection}{Subsections}
\Crefname{subsection}{Subsection}{Subsections}
\crefname{figure}{Figure}{Figures}
\Crefname{figure}{Figure}{Figures}

\newcommand{\Diff}{\textnormal{Diff}}

\newcommand{\TDiff}{\textnormal{TDiff}}

\newcommand{\Homeo}{\textnormal{Homeo}}

\newcommand{\spinc}{\textnormal{spin}^{c}}
\newcommand{\ab}{\mathrm{ab}}
\newcommand{\im}{\mathrm{Im}}
\newcommand{\Coker}{\mathrm{Coker}}

\newcommand{\SW}{\mathrm{SW}}
\newcommand{\FSW}{\mathrm{FSW}}

\newcommand{\id}{\textnormal{id}}
\newcommand{\Z}{\mathbb{Z}}
\newcommand{\N}{\mathbb{N}}
\newcommand{\R}{\mathbb{R}}

\newcommand{\CP}{\mathbb{CP}}
\newcommand{\fraks}{\mathfrak{s}}
\newcommand{\frakt}{\mathfrak{t}}

\newcommand{\calZ}{\mathcal{Z}}
\newcommand{\calO}{\mathcal{O}}

\newcommand{\del}{\partial}

\newcommand{\HMcheck}{\widecheck{H \! M}}
\newcommand{\HMhat}{\widehat{H \! M}}

\newcommand{\HMarrow}{\overrightarrow{H \! M}}

\title[Localization of groups of exotic diffeomorphisms]{On localizing groups of exotic diffeomorphisms of 4-manifolds}

\author{Hokuto Konno}
\address{Graduate School of Mathematical Sciences, the University of Tokyo, 3-8-1 Komaba, Meguro, Tokyo 153-8914, Japan}
\email{konno@ms.u-tokyo.ac.jp}

\author{Abhishek Mallick}
\address{Department of Mathematics, Rutgers University, Hill Center, Busch Campus, 110 Frelinghuysen Road
Piscataway, NJ 08854, USA}
\email{abhishek.mallick@rutgers.edu}

\begin{document}

\maketitle

\begin{abstract}
Ruberman in the 90's showed that the group of exotic diffeomorphisms of closed 4-manifolds can be infinitely generated.
We provide various results on the question of when such infinite generation can localize to a smaller embedded submanifold of the original manifold.
Our results include:
(1) All known infinitely generated groups of exotic diffeomorphisms of 4-manifolds detected by families Seiberg--Witten theory do not localize to any topologically (locally-flatly) embedded rational homology balls in the ambient 4-manifold.
(2) Many exotic diffeomorphisms cannot be obtained as Dehn twists along homology spheres (under mild assumptions). 
(3) There is no contractible 4-manifolds with Seifert fibered boundary that have a universal property for exotic diffeomorphisms analogous to a universal cork.
In addition, there is no universal compact 4-manifold $W$ such that the set of exotic diffeomorphisms of a 4-manifold can localize to an embedding of $W$. 
(4) Certain infinite generations of exotic diffeomorphism groups do localize to a non-compact subset $V$ with a small Betti number, but not to any compact subset of $V$. (5) An analogous result holds for mapping class groups of 4-manifolds.
\end{abstract}


\section{Introduction}
\label{section Main results}

Given a smooth manifold $X$, a diffeomorphism $f : X \to X$ is said to be {\it exotic} if $f$ is topologically isotopic to the identity but not smoothly.
4 is the smallest dimension where exotic diffeomorphisms exist. After a long hiatus following Ruberman's pioneering work \cite{Rub98,Rub99} exotic diffeomorphisms of 4-manifolds have attracted significant interest in recent studies through the advancement of families Seiberg--Witten theory
\cite{BK20gluing,KM20Dehn,Lin20Dehn,LinMukherjee21family,iida2022diffeomorphisms,konno2023exotic}.
One of the main results in Ruberman \cite{Rub99}  was to prove that the group of (components of) exotic diffeomorphisms can be infinitely generated for some 4-manifolds.

Given a diffeomorphism $f : X \to X$, a natural question is whether $f$ can ``localize" to a smaller region of $X$, i.e. whether the support of $f$ can be put into a smaller region by isotoping $f$.
We investigate this problem from the following angle: Can we localize an infinitely generated group of exotic diffeomorphisms of a 4-manifold to a smaller region?

The significance of such an approach lies in the fact that localization produces exotic diffeomorphisms of small 4-manifolds, while the known invariants in the literature used to detect exotic diffeomorphisms of 4-manifolds typically only work for 4-manifolds with complicated intersection forms. A success of this localization approach was observed in \cite{konno2023exotic}, where in a joint work with Taniguchi, the authors showed that some exotic diffeomorphisms of some 4-manifolds can localize to embedded contractible manifolds. This gave the first example of exotic diffeomorphisms of contractible 4-manifolds. On the other hand, as mentioned before, several authors such as Ruberman \cite{Rub98}, Baraglia \cite{Baraglia23mapping}, and the first author \cite{konno2023homology}, have shown that several exotic diffeomorphism groups and mapping class groups of closed 4-manifolds are infinitely generated. Hence a naturally arising follow-up question is whether infinitely generated groups of exotic diffeomorphisms of a closed 4-manifold can localize to a smaller region. This is the overarching question that motivates this article.

We demonstrate that all known infinitely generated groups of exotic diffeomorphisms of 4-manifolds detected by families Seiberg--Witten theory do \textit{not} localize even to fairly large embedded 4-manifolds, such as complements of certain embedded surfaces (\cref{subsection intro Localizability}).
This argument has various interesting consequences.
For example, we see that such infinitely generated groups do not localize to any (even topological) rational homology disk or homology cylinder in the ambient 4-manifold (\cref{cor: intro stabilized 4-manifolds homology disk}). This demonstrates that, broadly speaking, the `localization' procedure is incapable of deducing whether the diffeomorphism groups of smaller embedded pieces are infinitely generated or not.
A similar argument shows also that many exotic diffeomorphisms that are Seiberg--Witten analog of Ruberman's first exotic diffeomorphisms cannot be obtained as the Dehn twist a rational homology 3-sphere $Y$, for a fixed embedding of the cylinder $Y\times [0,1]$ into a 4-manifold (\cref{cor: intro Dehn twists}).

Following a slightly different argument, we also show that it is also not possible to localize any pair of exotic diffeomorphisms of a closed 4-manifold to a homology disk with Seifert fibered homology sphere boundary (Theorem~\ref{thm: simultaneous}). Coupled with the work of \cite{krushkal2024corks}, this gives the first example of an exotic diffeomorphism on a contractible 4-manifold, where the exotic diffeomorphism is not of the form of a Dehn twist along a Seifert fibered space, see Corollary~\ref{thm:non_dehn_twist}.

Another motivation of our work comes from exploring the existence of \textit{universal} objects governing the set of exotic diffeomorphisms of \textit{any} closed 4-manifolds. These questions are related to the existence of universal corks, posed by Akbulut \cite{AkbulutYasuiPlugs}. We show that there are no such universal corks for exotic diffeomorphism  with a Seifert fibered space as its boundary (Theorem~\ref{thm: C is not universal}). We also provide an exotic diffeomorphism analog of prior work by Tange \cite{TangeNon-existence23} and Yasui \cite{YasuiNonexistence19} (for ordinary corks) to show that there exists no universal compact 4-manifold $W$ such that the set of exotic diffeomorphisms of a 4-manifold can localize to an embedding of $W$ (\cref{cor: universal non-exitence}). 


Amidst all these non-localization results, we also demonstrate an example where it is possible to localize an infinite generation of the exotic diffeomorphism group of a closed 4-manifold to a relatively small embedded open manifold (\cref{thm: localization intro}). Moreover, we show that this infinite generation cannot localize to \textit{any} compact subset of the open manifold, even though all the diffeomorphisms concerned are compactly supported. Lastly, we also show that an analogous result holds for mapping class groups (\cref{thm: localization intro Diff}).

Our main technical tools stem from the families adjunction inequality by Baraglia~\cite{B202}. For various results, we also use the gluing formula to compute the families Seiberg--Witten invariant proved by Baraglia and the first author \cite{BK20gluing} and Lin \cite{lin2022family}.

\subsection{Notations and definitions}

Before stating our results, we find it useful to fix our notation of groups of exotic diffeomorphisms. 
For a smooth 4-manifold $X$, let $\Diff(X)$ denote the group of diffeomorphisms equipped with the $C^\infty$-topology.
Let $\TDiff(X)$ denote the {\it Torelli diffeomorphism group}, defined as the subgroup of $\Diff(X)$ that acts trivially on the homology $H_\ast(X;\Z)$.
The component group $\pi_0(\TDiff(X))$ is called the {\it Torelli group}.
If $X$ is simply-connected, $\mathrm{\bf{T}}\Diff$ also indicates {\bf T}opologically trivial diffeomorphisms: a result by Quinn~\cite{Q86} and Perron~\cite{P86} implies that
\[
\pi_0(\TDiff(X)) = \ker(\pi_0(\Diff(X)) \to \pi_0(\Homeo(X)))
\]
for simply connected and closed 4-manifold $X$ (see also recent work \cite{gabai2023pseudoisotopies} for a correction of a gap in Quinn's paper).
Namely, the Torelli group is the group of (components of) exotic diffeomorphisms.

Similarly, if $\del X \neq  \emptyset$, let $\Diff_\del(X)$ denote the group of diffeomorphisms that are the identity near $\del X$.
The {\it relative Torelli diffeomorphism group} is defined to be $\TDiff_\del(X) = \Diff_\del(X) \cap \TDiff(X)$.
If $\del X$ is a rational homology sphere and $X$ is simply-connected, work by Orson--Powell~\cite[Theorem A]{OrsonPowell2022} implies that 
\[
\pi_0(\TDiff_\del(X)) = \ker(\pi_0(\Diff_\del(X)) \to \pi_0(\Homeo_\del(X))).
\]

Given an open manifold $C$, we denote by $\Diff_c(C)$ the group of diffeomorphisms with compact support and set $\TDiff_c(C) = \Diff_c(C)\cap \TDiff(C)$.
If $C$ is an open submanifold of $X$, we have natural inclusion maps $\Diff_c(C) \hookrightarrow \Diff(X)$ and $\TDiff_c(C) \hookrightarrow \TDiff(X)$.
We shall consider the induced map $\pi_0(\TDiff_c(C)) \to \pi_0(\TDiff(X))$.

For non-abelian groups, a finitely generated group can have an infinitely generated subgroup, so it is not obvious how to measure the `size' of groups.
One reasonable way is to pass to abelianizations and use ranks, which we adopt in this paper.
We denote by the induced map between the abelianizations by  $\pi_0(\TDiff_c(C))_{\ab} \to \pi_0(\TDiff(X))_{\ab}$.

Similarly, if $C$ is a 
compact codimension-0 submanifold of $X$, we have natural inclusion maps $\Diff_\del(C) \hookrightarrow \Diff(X)$ and $\TDiff_\del(C) \hookrightarrow \TDiff(X)$.

Lastly, to state our (non-)localization results, we make the following definition:
\begin{defi}
\label{defi: localization}
Let $X$ be a smooth closed 4-manifold.
Given a subgroup $\calZ \subset \pi_0(\TDiff(X))_\ab$ and a subset $C$ of $X$, we say that $\calZ$ {\it localizes to $C$} if, for every $z \in \calZ$, there is a representative $f \in \TDiff(X)$ of $z$ such that the support of $f$ is contained in $\mathrm{Int}(C)$.
\end{defi}

If $C$ is an open submanifold of $X$, $\calZ$ localizes to $C$ if and only if
\[
\calZ \subset \im(\pi_0(\TDiff_c(C))_\ab \to \pi_0(\TDiff(X))_\ab).
\]
Similarly, if $C$ is compact submanifold of codimension-0, $\calZ$ localizes to $C$ if and only if
\[
\calZ \subset \im(\pi_0(\TDiff_\del(C))_\ab \to \pi_0(\TDiff(X))_\ab).
\]

\subsection{Non-localization to homology disks}

Now we are ready to state our first result, which tells that some infinitely generated group of exotic diffeomorphisms cannot localize to {\it any} embedded homology 4-disk, if a 4-manifold is appropriately stabilized.
Here we can handle even a topological object: $C$ is said to be a {\it topological rational homology 4-disk} if $C$ is a compact topological manifold with $H_{\ast}(C;\mathbb{Q})$ isomorphic to $H_{\ast}(D^4;\mathbb{Q})$.
Unless otherwise stated, a topological embedding of a topological manifold is assumed to be locally flat in this paper.
We have:

\begin{thm}
\label{cor: intro stabilized 4-manifolds homology disk}
Given a simply-connected closed smooth 4-manifold $X$, there is $N \geq 0$ with the following property: set $X' := X\#N(S^2\times S^2)$.
Then there exists a $\Z^\infty$-summand $\calZ$ of  $\pi_0(\TDiff(X'))_{\ab}$ such that any infinite-rank subgroup of $\calZ$ does not localize to $C$,
for any topological rational homology 4-disk $C$ embedded in $X'$.

In particular, if $C$ is smooth and smoothly embedded into $X'$,
\begin{align*}
\Coker(\pi_0(\TDiff_\del(C))_\ab \to \pi_0(\TDiff(X'))_\ab)
\end{align*}
contains a $\Z^\infty$-summand for any such $C$.
\end{thm}

\begin{ex}
\label{ex: homoplogy disk}
For 
$X = E(n)\#k\overline{\CP}^2$ with $n\geq2$, $k\geq 0$, 
it shall turn out that $N=1$ works (\cref{ex: intro reduction}).
Thus the conclusion of
\cref{cor: intro stabilized 4-manifolds homology disk} holds for
\[
X' = E(n)\#k\overline{\CP}^2\#S^2\times S^2
\]
(which dissolves as is well-known).
\end{ex}
\noindent
Theorem~\ref{cor: intro stabilized 4-manifolds homology disk} is in fact a consequence of a much more general result obtained in Theorem~\ref{thm: intro stabilized 4-manifolds} regarding non-localization to surface complements.

\subsection{Exotic diffeomorphisms that cannot be certain Dehn twists}

Next, we discuss non-localization to homology cylinders.
To put this in context, recall that there are qualitatively different types of exotic diffeomorphisms that a 4-manifold can support.
The first type is Ruberman's first example of exotic diffeomorphism \cite{Rub98}, and its variants detected by families Seiberg--Witten theory  \cite{BK20gluing,iida2022diffeomorphisms}.
Another variant stems from Dehn twists along Seifert fibered 3-manifolds \cite{KM20Dehn,Lin20Dehn,konno2023exotic}. In particular, in our previous work \cite{konno2023exotic} with Taniguchi, we proved that the Dehn twist along various Seifert fibered homology 3-spheres $Y$ can give exotic diffeomorphisms of a 4-manifold $X$ in which $Y$ is embedded.
Given this situation, it is natural to ask if a given exotic diffeomorphism is obtained as a Dehn twist along some 3-manifold that admits a $S^1$-action.

The exotic diffeomorphisms constructed in this paper are, in fact, of Ruberman-type, and our non-localization result ensures that many of such exotic diffeomorphisms cannot simultaneously be obtained as Dehn twists along embedded contractible manifolds.
To state this, note two obvious criteria for a Dehn twist along a 3-manifold $Y$ following from the definition. Firstly, it is supported in a tubular neighborhood $\nu(Y)$ of $Y$, which is diffeomorphic to a cylinder $Y \times [0,1]$. Moreover, the only $3$-manifolds that can support an exotic Dehn twist are Seifert fibered spaces. 

Generally, we say that a smooth compact 4-manifold $W$ is a {\it rational homology cylinder} if $W$ is a rational homology cobordism from some rational homology 3-sphere to another rational homology 3-sphere. Our first result obstructs exotic diffeomorphisms from being a Dehn twist on a cylinder with a fixed embedding, by showing localization to a topological rational homology cylinder is not always possible:

\begin{thm}
\label{cor: intro Dehn twists}
Let $X'$ be as in \cref{cor: intro stabilized 4-manifolds homology disk}.
Then there exists a $\Z^\infty$-summand $\calZ$ of  $\pi_0(\TDiff(X'))_{\ab}$ such that any infinite-rank subgroup of $\calZ$ does not localize to $W$,
for any topological rational homology cylinder $W$ embedded in $X'$.

In particular, if $W$ is smooth and smoothly embedded into $X'$,
\begin{align*}
\Coker(\pi_0(\TDiff_\del(W))_\ab \to \pi_0(\TDiff(X'))_\ab)
\end{align*}
contains a $\Z^\infty$-summand for any such $W$.
\end{thm}

\begin{ex}
As in \cref{ex: homoplogy disk}, the conclusion of
\cref{cor: intro Dehn twists} holds for
$X' = E(n)\#k\overline{\CP}^2\#S^2\times S^2$
with $n\geq2$, $k\geq 0$
(\cref{ex: intro reduction}).
\end{ex}

While Theorem~\ref{cor: intro Dehn twists} deals with non-localization of an infinite rank summand, one may also seek an analogous result for finitely many exotic diffeomorphisms. To appreciate this result, let us first state a shortcoming of Theorem~\ref{cor: intro Dehn twists} through the lens of the above question. Although, Theorem~\ref{cor: intro Dehn twists} guarantees that, for any embedding of a rational homology cylinder $W$ into $X'$, any infinite rank summand of $\mathcal{Z}(\alpha)$ do not localize to $W$, it does not give us an explicit set of members of $\mathcal{Z}(\alpha)$ that will never localize to any rational homology disk/cylinder embedding.

 Indeed, we can ask whether it is always possible to localize any finitely many exotic diffeomorphisms as Dehn twists (or generally as an exotic diffeomorphism) on any homology disk.

Our result below shows that such a localization is not possible even for a pair of exotic diffeomorphisms:
\begin{thm}\label{thm: simultaneous}
There exists a simply-connected closed smooth 4-manifold $X$ and a pair of exotic diffeomorphisms $f_1, f_2$ of $X$ with the following properties:
there exists no compact smooth integer homology 4-disk $W$ smoothly embedded in $X$ such that both $f_1, f_2$ localize to $W$ and that $\del W$ is a Seifert fibered space.
\end{thm}
In particular, the conclusion of Theorem~\ref{thm: simultaneous} implies that $f_1$ and $f_2$ cannot simultaneously be exotic Dehn twist on \textit{any} embedding of \textit{any} homology 4-disk to $X$. In \cite{krushkal2024corks} Krushkal-Mukherjee-Powell-Warren proved an interesting result that allows localization of certain exotic diffeomorphisms to contractible manifolds. Theorem~\ref{thm: simultaneous} coupled with results from \cite{krushkal2024corks} gives us the following:
\begin{cor}\label{thm:non_dehn_twist}
There exist exotic diffeomorphisms on contractible manifolds that are not isotopic to boundary Dehn-twists along Seifert fibered homology spheres.
\end{cor}
Corollary~\ref{thm:non_dehn_twist} gives the first example of an exotic diffeomorphism on a contractible manifold that is not of the form of a Dehn twist. Indeed, prior to this, all explicitly known exotic diffeomorphisms on contractible manifolds were in the form of a Dehn-twist along a Seifert fibered space. In fact, it follows from the proof of Theorem~\ref{thm: simultaneous} that we can also generalize the statement of Corollary~\ref{thm:non_dehn_twist} by showing that there are exotic diffeomorphisms on contractible manifolds $C$ such that $\partial C$ is not a connected sum of Seifert fibered spaces (such that all of them oriented in a similar manner, see below).

Finally, we show that up to fixing an orientation convention of the boundaries, it is also possible to give an explicit set of infinitely many exotic diffeomorphisms which does not localize as Dehn twists on homology disks. Let us first clarify the orientation convention:

Note that Seifert fibered integer homology spheres admit a canonical orientation as a boundary of the negative definite plumbing associated to it. Let $\calO$ indicate the choice of either the canonical orientations for all Seifert fibered 3-manifolds, or the opposite orientation to the canonical orientations for all Seifert fibered 3-manifolds.
Namely, $\calO$ is a simultaneous choice of orientations of all Seifert fibered 3-manifolds, either the canonical ones or the opposite ones. We show the following:

\begin{thm}
\label{thm: not dehn twist fixed orientation}
Fix $\calO$.
Then there exists a simply-connected oriented closed smooth 4-manifold $X$ and infinitely many exotic diffeomorphisms $\{f_i\}_{i=1}^\infty$ of $X$ spanning a  $\Z^\infty$-summand of $\pi_0(\TDiff(X))_{\ab}$ with the following properties:

For each $i$, there exists no oriented compact integral homology 4-disk $W$ such that:
\begin{itemize}
\item[(a)] There is an orientation-preserving smooth embedding $W \hookrightarrow X$ along which $f_i$ localizes to $W$.
\item[(b)] $\del W$ is a Seifert fibered 3-manifold oriented by $\calO$.
\end{itemize}

\end{thm}

\subsection{Non-existence of universal support for exotic diffeomorphism}\label{universal_intro}
Our next result is related to obstructing the existing the existence of `universal support' for exotic diffeomorphisms. 

A cork is a pair $(C,\tau)$ of contractible 4-manifold with boundary and a diffeomorphism $\tau: \partial C \rightarrow \partial C$. By the work of \cite{curtis1996decomposition, matveyev1996decomposition}, any two exotic diffeomorphisms of a closed simply-connected 4-manifold are related by a cork-twist. The first example of a cork was found by Akbulut \cite{akbulut1991fake}, and since then many other examples of corks have appeared in the literature \cite{AkbulutYasuiPlugs, dai2022corks}. Let us now recall the notion of a universal cork:

\begin{defi}
We say that a cork $(C,\tau)$ is a {\it universal cork} if \textit{any} exotic pair $X_0, X_1$ of simply-connected closed smooth 4-manifolds is related by a cork twist of $(C,\tau)$ along an orientation preserving embedding of $C$ into $X_i$.
\end{defi} 
Whether there exists a universal cork or not was first posed by Akbulut, see \cite{AkbulutYasuiPlugs}. There have been several non-existence of universal object type results in the literature, under some assumptions. These began with the results by Tange \cite{TangeNon-existence23} and its subsequent generalization by Yasui \cite{YasuiNonexistence19}. In \cite{TangeNon-existence23} Tange showed that there are infinitely many exotic smooth structures on a particular closed 4-manifold such that a fixed cork twist cannot realize all of them. This was particularly interesting since before this Gompf \cite{Gompf_infinite} had discovered infinite order corks. Yasui \cite{YasuiNonexistence19} generalized this work by showing there are no universal 4-manifold $W$, such that regluing $W$ by a boundary diffeomorphism can realize every exotic smooth structure of closed 4-manifold. Building on an observation from \cite[Theorem 8.6]{LRS18}, in \cite{ladu2024akbulut} Ladu showed that the Akbulut cork is not universal.

It is then natural to ask for an exotic diffeomorphism version of the above story. Indeed, drawing a comparison with the corks, there are known examples of exotic diffeomorphism of a closed 4-manifold that can be localized to a contractible manifold. The first examples of such were given as Mazur manifolds bounded by Seifert fibered spaces $\Sigma(2,3,13)$ and $\Sigma(2,3,25)$ by the authors with Taniguchi \cite{konno2023exotic}. These were defined as $\pi_1$-corks in \cite{konno2023exotic} keeping with the resemblance of such Mazur manifolds (equipped with Dehn twist) with corks.

Maintaining the flow of the cork-story, in this article, we propose an exotic diffeomorphism version of `universal cork' type objects.
\begin{defi}
Let $W$ be an oriented compact smooth 4-manifold with boundary.
We say that $W$ is {\it diff-universal} if, for every oriented closed simply-connected smooth 4-manifold $X$ and every exotic diffeomorphism $f : X \to X$, there is an orientation-preserving smooth embedding $i : W \hookrightarrow X$ along which $f$ localizes to $W$.
If further $W$ is contractible, we say that $W$ is a {\it universal families cork}.
\end{defi}

Since we know the existence of exotic diffeomorphism of contractible manifolds, an analog of the question posed by Akbulut, for exotic diffeomorphisms comes naturally:

\begin{ques}\label{diff-universal_ques_intro}
Are there any familes universal corks?  More generally, are there any diff-universal 4-manifold?
\end{ques}
Much like the theory for corks, the authors expect a negative answer to the above question. In this article, we make progress towards a negative answer to Question~\ref{diff-universal_ques_intro} on several fronts.
\begin{thm}
\label{thm: C is not universal}
Let $C$ be a compact oriented smooth integer homology disk bounded by a Seifert fibered 3-manifold. Then $C$ and $-C$ are not universal families corks.
\end{thm}
It follows from this theorem that the Mazur manifolds that were shown to support exotic diffeomorphisms in \cite{konno2023exotic} are not universal. 
\begin{rmk}
We record the following remarks pertaining to a somewhat straightforward generalization of Theorem~\ref{thm: C is not universal}.
\begin{enumerate}
\item Theorem~\ref{thm: C is not universal} follows from a much more general Theorem~\ref{thm: Seiferts do not give universal families cork}, where we allow $C$ from Theorem~\ref{thm: C is not universal} to have almost rationally plumbed integer homology spheres as boundary.

\item It is interesting to compare Theorem~\ref{thm: C is not universal} with the conventional theory for corks. For instance, there is no known example of a cork bounded by Seifert fibered homology spheres. Hence Theorem~\ref{thm: C is not universal} is not quite meaningful in the usual cork theory.

\item Moreover, somewhat dual to above, there are instances where we can also allow $C$ to have boundary that is not an AR plumbed homology spheres. For example, an immediate consequence of our approach is  that the Akbulut cork and the Stevedore cork are not universal families corks, see Theorem~\ref{thm:akbulut_not_families_universal}. However, it is not known whether such contractible manifolds can support exotic diffeomorphisms.

\end{enumerate}
 
\end{rmk}

We now provide an exotic diffeomorphism analog of Yasui's result. For exotic diffeomorphisms, a corresponding question is if there is a `universal support', i.e. some universal 4-manifold $W$ such that any closed 4-manifold $X$ admits an embedding of $W$ along which $\pi_0(\TDiff(X))$ localizes to $W$.
We prove such $W$ does not exist, even if allowing topological for embeddings:



\begin{thm}
\label{cor: universal non-exitence}  
There exists no compact topological 4-manifold $W$ with the following property: for any simply-connected closed smooth 4-manifold $X$, there exists a topological embedding $W \hookrightarrow X$ along which $\pi_0(\TDiff(X))$ localizes to $W$.
\end{thm}

\cref{cor: universal non-exitence} is a families analog of \cite[Corollary 1.5]{YasuiNonexistence19}.
This is a rather immediate consequence of the following much stronger result, which is a counterpart of \cite[Theorem 1.3]{YasuiNonexistence19}:

\begin{thm}
\label{thm: universal non-exitence general}
Given $m,n>0$, there exists a simply-connected closed smooth 4-manifold $X$ with the following property: for any compact codimension-0 topological submanifold $W$ of $X$ with $b_2(W) < m$ and $b_1(\del W)<n$,
there exists a $\Z^\infty$-summand $\calZ(W)$ of
$\pi_0(\TDiff(X))_\ab$ such that any infinite-rank subgroup of $\calZ(W)$ does not localize to $W$.
\end{thm}

\subsection{Compact vs. Non-compact}

Next, we give an example of infinite generation on a closed 4-manifold that localizes to an open submanifold with a much smaller Betti number, but not to any compact subset of it.
This illustrates that even if some infinitely generated group can localize to a smaller piece, compactness can be crucial:

\begin{thm}
\label{thm: localization intro}
Let $X=K3 \# \overline{\mathbb{CP}}^2$ and $X' = X\#S^2\times S^2$. Then there exist a $\Z^\infty$-summand $\mathcal{Z}$ of $\pi_0(\TDiff(X'))_\ab$ and a contractible open submanifold $C \subset X$ such that:
\begin{itemize}
\item[(i)] 
$\mathcal{Z}$ localizes to $C \# S^2 \times S^2 \subset X'$, but
\item[(ii)] for any compact subset $K$ of $C\#S^2\times S^2$, any infinite-rank subgroup of $\calZ$ does not localize to $K$.
\end{itemize}
\end{thm}

The open contractible 4-manifold $C$ is constructed in a fairly explicit way out of Akbulut corks.
The 4-manifold $X'=K3\#\overline{\CP}^2\#S^2\times S^2$ is the smallest 4-manifold for which the Torelli group is known to be infinitely generated \cite{Rub99}, while one can see a slightly smaller 4-manifold $K3\#S^2\times S^2$ actually satisfies this property by a similar argument\footnote{This can be deduced from \cite{konno2023homology}. The argument is summarized in \cref{thm: infinite generation general}.}.

\begin{rmk}
The proof of Theorem~\ref{thm: localization intro} factors through proving the following result: 
For every $n>0$, there exists a compact smooth 4-manifold $C$ with $b_2(C)=2$ such that $\pi_0(\TDiff_\del(C))_\ab$ contains a $\Z^n$-summand.
\end{rmk}

\subsection{Non-localization of $\pi_0(\Diff(X))$}

We also give a result on non-localization for $\pi_0(\Diff(X))_{\ab}$, instead of $\pi_0(\TDiff(X))_{\ab}$.
We use the term localization as in \cref{defi: localization} also for a subgroup of $\pi_0(\Diff(X))_{\ab}$.
Recently, the first author \cite{konno2023homology} and Baraglia~\cite{Baraglia23mapping} gave the first examples of simply-connected closed 4-manifolds with infinitely generated mapping class groups.
We see that a rather small open 4-manifold can admit infinitely generated (abelianized) mapping class group with a non-localization property to compact subsets:

\begin{thm}
\label{thm: localization intro Diff}
There exist a contractible open smooth 4-manifold $C$ and a $(\Z/2)^\infty$-summand $\mathcal{Z}$ of $\pi_0(\Diff_c(C'))_\ab$ for $C'=C\#S^2\times S^2$, such that, for any compact subset $K$ of $C'$, any infinitely generated subgroup of $\calZ$ does not localize to $K$.
\end{thm}

\begin{rmk}
Gompf \cite{Gompf_diffeo_R4} proved the existence of an open 4-manifold with trivial intersection form (actually an exotic $\R^4$) that has infinitely generated mapping class group.
However, this is substantially different from our situation, since the mapping class group considered in \cite{Gompf_diffeo_R4} is not the compactly supported mapping class group as opposed to \cref{thm: localization intro Diff}, and the infinite generation in \cite{Gompf_diffeo_R4} comes from diffeomorphisms that are non-trivial near the end. 
\end{rmk}

\subsection{Non-localization to surface complements}
\label{subsection intro Localizability}

Many of the above results are, in fact, consequences of the following non-localization.
This result says that some infinitely generated group of exotic diffeomorphisms cannot localize even to the complement of any homologically non-trivial surface, if a 4-manifold is appropriately stabilized.
Precisely, we shall prove:

\begin{thm}
\label{thm: intro stabilized 4-manifolds}
Let $X$ be a simply-connected closed smooth 4-manifold.
Then there is $N \geq 0$ with the following property: set $X' := X\#N(S^2\times S^2)$. For every non-zero homology class $\alpha \in H_2(X';\Z)$, there exists a $\Z^\infty$-summand $\calZ(\alpha)$ of  $\pi_0(\TDiff(X'))_{\ab}$ such that any infinite-rank subgroup of $\calZ(\alpha)$ does not localize to $X'\setminus \Sigma$,
for any oriented, closed, connected, and smoothly embedded surface $\Sigma \subset X'$ that represents  $\alpha$.

Furthermore, we can take $\calZ(\alpha)$ to depend only on the ray $\Z\cdot \alpha \subset H_2(X';\Z)$, i.e. $\calZ(\alpha)=\calZ(n\alpha)$ for any $\alpha \neq 0$ and $n \in\Z\setminus\{0\}$.
\end{thm}

The non-localization of all infinite-rank subgroups of $\calZ(\alpha)$ in \cref{thm: intro stabilized 4-manifolds} is equivalent to that the intersection
\begin{align}
\label{eq: Z cap intro}
\calZ(\alpha) \cap \im(\pi_0(\TDiff_c(X'\setminus \Sigma))_\ab \to \pi_0(\TDiff(X'))_\ab)
\end{align}
is finitely generated.

As an obvious consequence of \cref{thm: intro stabilized 4-manifolds}, we have:

\begin{cor}
\label{cor: Coker intro}
Let $X'$ be as in \cref{thm: intro stabilized 4-manifolds}.   
Then, for every orientable, closed, connected, homologically non-trivial, and smoothly embedded surface $\Sigma \subset X'$, 
\begin{align*}
\Coker(\pi_0(\TDiff_c(X'\setminus \Sigma))_\ab \to \pi_0(\TDiff(X'))_\ab)
\end{align*}
contains a $\Z^\infty$-summand.
\end{cor}

Although the number of stabilizations needed in \cref{thm: intro stabilized 4-manifolds} is in general unknown, we see that one stabilization is enough for many examples.
This covers all known examples of 4-manifolds for which infinite generation of the group of exotic diffeomorphisms of 4-manifolds was proven by families Seiberg--Witten theory:

\begin{thm}
\label{thm: intro stabilized 4-manifolds non-negative}
Let $X$ be a simply-connected closed smooth 4-manifold with $b^+(X)\geq2$.
Suppose that there is a spin$^c$ structure $\fraks$ with formal dimension 0 and $\SW(X,\fraks) \neq 0$ mod 2.
Suppose further that $X$ contains the nucleus $N(2)$.
Set $X' := X\#S^2\times S^2$.
Then, for any non-zero homology class $\alpha \in H_2(X';\Z)$ with $\alpha^2\geq 0$, there exists a $\Z^\infty$-summand $\calZ(\alpha)$ of  $\pi_0(\TDiff(X'))_{\ab}$ such that any infinite-rank subgroup of $\calZ(\alpha)$ does not localize to $X'\setminus \Sigma$,
for any oriented, closed, connected, and smoothly embedded surface $\Sigma \subset X'$ that represents  $\alpha$.
Furthermore, we can take $\calZ(\alpha)$ to depend only on the ray $\Z\cdot \alpha \subset H_2(X';\Z)$.
\end{thm}

Thanks to \cref{thm: intro stabilized 4-manifolds non-negative}, we can give explicit examples of the ambient 4-manifolds $X'$ for which we have non-localization results:

\begin{ex}
\label{ex: elliptic surface}
For $n\geq 2$, 
$X = E(n)$ satisfies this assumption of \cref{thm: intro stabilized 4-manifolds non-negative}.
Indeed, $K3=E(2)$ contains disjoint three copies of $N(2)$, and $E(n)$ for $n>2$ is constructed as a fiber sum of $E(2)$ and $E(n-2)$ along a torus in one of $N(2)$.
Thus $E(n)$ contains (at least two copies of) $N(2)$.
\end{ex}

Non-localization to any topological rational homology disks and homology cylinders holds for this once-stabilized 4-manifold (see \cref{section Proof of non-localizability} for the proofs):

\begin{cor}
\label{cor: intro stabilized 4-manifolds non-negative}
Let $X'$ be as in \cref{thm: intro stabilized 4-manifolds non-negative}.  Then the conclusions of
\cref{cor: intro stabilized 4-manifolds homology disk} (non-localization to homology disks) and \cref{cor: intro Dehn twists} (non-localization to homology cylinders) hold for $X'$.
\end{cor}

\begin{ex}
\label{ex: intro reduction}
It follows from \cref{cor: intro stabilized 4-manifolds non-negative} and \cref{ex: elliptic surface} that the conclusions of
\cref{cor: intro stabilized 4-manifolds homology disk,cor: intro Dehn twists} hold for
\[
X' = E(n)\#k\overline{\CP}^2\#S^2\times S^2
\]
with $n\geq2$, $k\geq 0$.
\end{ex}

\subsection{Other remarks}

\begin{rmk}
All exotic diffeomorphisms that appear in this paper are trivial after a single stabilization by $S^2 \times S^2$.
In fact, the $\Z^\infty$-summand $\calZ(\alpha)$ in \cref{thm: intro stabilized 4-manifolds,thm: intro stabilized 4-manifolds non-negative} we shall construct is generated by exotic diffeomorphisms $\{f_i\}_{i=1}^\infty$ for which each $f_i$ is smoothly isotopic to the identity of $X'\#S^2\times S^2$ (after extending $f_i$ to $X'\#S^2\times S^2$ by isotopy. See \cite[Theorem 5.3]{auckly2015stable} for this extension).
\end{rmk}

\begin{rmk}
\label{rmk: vs. families gluing}
Our main technical tools are families adjunction inequality by Baraglia \cite{B202} and the families gluing theorems by Baraglia and the first author \cite{BK20gluing} and Lin~\cite{lin2022family}. We emphasize that for certain non-localization theorems such as Theorem~\ref{cor: intro stabilized 4-manifolds homology disk} and Theorem~\ref{thm: intro stabilized 4-manifolds} we use families adjunction inequality, while it is also possible to approach them using the families gluing theorems. Indeed, this leads to certain salient features of our results that seem somewhat out of reach using the if one used families gluing instead. We highlight some of those features below:
\begin{enumerate}
\item As we have seen, many of our results obstruct localization even to topological submanifolds.
On the other hand, a gluing argument works only for smooth objects.
\item The proof of \cref{thm: intro stabilized 4-manifolds} gives a quantitative estimate on which diffeomorphisms belonging to $\calZ(\alpha)$ do not localize to a surface complement. This remark applies to several other similar results in this paper.
A gluing argument seems to require the use of abstract finiteness on the number of spin$^c$ structures with non-zero cobordism maps, which gives no explicit estimate.
\end{enumerate}
\end{rmk}


\noindent\textbf{Organization.}
This article is organized as follows.
In \cref{section Families Seiberg--Witten invariant}, we summarize tools from families Seiberg--Witten theory that shall be used to prove various non-localization results.
In \cref{section localization}, we shall prove the localization result, (i) of Theorem~\ref{thm: localization intro}, which shall be also an ingredient of \cref{thm: localization intro Diff}.
All non-localization results are proven in \cref{section Proof of non-localizability}, except for certain non-existence results for universal families corks, and exotic Dehn twist which are proven in 
\cref{senction Proof of the results on universal families corks}.

\noindent\textbf{Acknowledgment.}
It is our pleasure to thank Masaki Taniguchi for numerous helpful discussions which helped to shape this article. 
We are also grateful to Anubhav Mukherjee and Kristen Hendricks for their comments on an earlier version of the paper and Irving Dai for helpful discussions. We also thank Mark Powell for motivating conversations and the authors of \cite{krushkal2024corks} for freely sharing their results. We thank Rutgers University-New Brunswick and MIT for hosting HK and AM respectively for seminars where this article was initiated. HK was partially supported by JSPS KAKENHI Grant Number 21K13785. AM was partially supported by NSF DMS-2019396. AM also thanks IMPAN, Poland; where part of the work was done.

\section{Families Seiberg--Witten invariant}
\label{section Families Seiberg--Witten invariant}

\subsection{Basics}
\label{subsection Families Seiberg--Witten invariant}

First, we recall the basics of the families Seiberg--Witten invariant (see such as \cite{LiLiu01,Rub98,BK20gluing} for details).
Let $X$ be a closed oriented smooth 4-manifold with $b^+(X) \geq 3$.
Let $\fraks$ be a spin$^c$ structure on $X$ with $d(\fraks)=-1$, where
$d(\fraks)$ denotes the formal dimension of the Seiberg--Witten moduli space:
\[
d(\fraks)
= \frac{1}{4}(c_1(\fraks)^2-2\chi(X)-3\sigma(X)).
\]
Let $\Diff(X,\fraks)$ denote the group of orientation-preserving diffeomorphisms that preserve the isomorphism class of $\fraks$.
For each $f \in \Diff(X,\fraks)$, we can define the {\it families Seiberg--Witten invariant}
\[
\FSW(X,\fraks,f) \in \Z/2.
\]

The quantity $\FSW(X,\fraks,f)$ is defined by counting the parameterized moduli space: one can consider a family of Seiberg--Witten equations on the mapping torus $E_f \to S^1$ of $f$, by picking a fiberwise metric and fiberwise perturbation on $E_f$.
The parameterized moduli space of the solutions to the parameterized equations is generically a compact 0-dimensional manifold, so we can count it in mod 2 and get a value in $\Z/2$.
The condition $b^+(X) \geq 3$ guarantees that the count is independent of choices, and we can get a well-defined invariant $\FSW(X,\fraks,f) \in \Z/2$.

The invariant $\FSW$ 
satisfies the additivity under compositions (\cite[Lemma 2.6]{Rub98}) and the isotopy invariance (\cite[Lemma 2.7]{Rub98}), and as a result we obtain a homomorphism
\begin{align}
\label{eq: homomorphism Diff sO mod 2}
\FSW(X,\fraks,-) : \pi_0(\Diff(X,\fraks)) \to \Z/2.
\end{align}

Further, we can upgrade $\FSW(X,\fraks,f)$ to a $\Z$-valued invariant in the following situation.
Let $\calO$ be a homology orientation, i.e. an orientation of the vector space $H^1(X;\R) \oplus H^+(X)$.
Here $H^+(X;\R)$ denotes a maximal-dimensional positive-definite subspace of $H^2(X;\R)$, which is known to be unique up to isotopy.
Let $\Diff(X,\fraks,\calO)$ denote the subgroup of $\Diff(X,\fraks)$ that preserves homology orientation.
If $f$ lies in $\Diff(X,\fraks,\calO)$, the parameterized moduli space is canonically oriented, so that we can get
\[
\FSW(X,\fraks,f) \in \Z
\]
by counting the parameterized moduli space over $\Z$.
Again, this count is independent of choices thanks to $b^+(X) \geq 3$.
Strictly speaking, the sign of $\FSW(X,\fraks,f)$ depends on a choice of $\calO$, but we usually drop $\calO$ from our notation.
If $f$ preserves the homology orientation of $X$, we interpret the invariant $\FSW(X,\fraks,f)$ as a $\Z$-valued invariant unless otherwise stated.

Additivity under compositions holds over $\Z$, thus we obtain a homomorphism
\begin{align}
\label{eq: homomorphism Diff sO}
\FSW(X,\fraks,-) : \pi_0(\Diff(X,\fraks,\calO)) \to \Z,
\end{align}
which descends to \eqref{eq: homomorphism Diff sO mod 2} if we pass to $\pi_0(\Diff(X,\fraks))$.

Set $\TDiff(X,\fraks) = \TDiff(X)\cap \Diff(X,\fraks)$.
If $X$ is simply-connected so that a spin$^c$ structure is determined by its first Chern class, we have $\TDiff(X,\fraks) = \TDiff(X)$.
Also, since we have $\TDiff(X,\fraks) \subset \Diff(X,\fraks,\calO)$, we can get a $\Z$-valued invariant $\FSW(X,\fraks,f) \in \Z$ for $f \in \TDiff(X,\fraks)$.
Restricting \eqref{eq: homomorphism Diff sO}, we get a homomorphism
\begin{align}
\label{eq: homomorphism TDiff}
\FSW(X,\fraks,-) : \pi_0(\TDiff(X,\fraks)) \to \Z.
\end{align}

\subsection{Properties}
\label{subsection Families Seiberg--Witten invariant}

Now we collect a few properties of families Seiberg--Witten invariants we shall use.

The first property is the finiteness on families basic classes:

\begin{pro}
\label{prop: finiteley many families basic classes}
Let $X$ be a smooth closed oriented 4-manifold with $b^+(X) \geq 3$.
Given an orientation-preserving diffeomorphism $f$ of $X$,
there are at most finitely many spin$^c$ structures $\fraks$ with $d(\fraks)=-1$ such that $\FSW(X,\fraks,f)\neq0$.
\end{pro}

\begin{proof}
This is a standard compactness argument.
The families Seiberg--Witten invariant $\FSW(X,\fraks,f)$ is the count of the parameterized moduli space for $\fraks$ and for a generic path $\{(g_t,\mu_t)\}_{t \in [0,1]}$ of metrics $g_t$ and $g_t$-self-dual 2-forms $\mu_t$ with $f^\ast(g_0,\mu_0)=(g_1,\mu_1)$.
Since this path is compact, just as in the proof of finiteness of basic classes for the usual Seiberg--Witten invariant (e.g. \cite[Theorem 5.2.4]{Mo96}), we have that there are at most finitely many spin$^c$ structures $\fraks$ on $X$ with a given formal dimension for which the parameterized moduli spaces for the path $\{(g_t,\mu_t)\}_{t \in [0,1]}$ are non-empty.
\end{proof}

The next property is a gluing formula proven by Baraglia and the first author \cite{BK20gluing}.
Let us summarize the setup.
Let $M$ be an oriented closed smooth 4-manifold with $b^+(M) \geq 2$ and $b_1(M)=0$.
Let $\fraks_M$ be a spin$^c$ structure on $M$ with $d(\fraks_M)=0$.
Let $\frakt$ denote the unique spin structure on $S^2 \times S^2$.
Let
\begin{align}
\label{eq: ingredient fnode}
f_0 \in \Diff_\del(S^2\times S^2 \setminus \mathrm{Int}(D^4))
\end{align}
be a relative diffeomorphism that satisfies $f_0^\ast=-1$ on $H^+(S^2 \times S^2) = \R(\mathrm{PD}[S^2 \times \{\mathrm{pt}\}]+\mathrm{PD}[\{\mathrm{pt}\} \times S^2]) \subset H^2(S^2\times S^2)$.
Set $(X,\fraks) := (M,\fraks_M)\#(S^2\times S^2,\frakt)$.
Note that $\id_{M}\#f_0 \in \Diff(X,\fraks)$ and $d(\fraks)=-1$, so we can define $\FSW(X,\fraks,\id_{M}\#f_0) \in \Z/2$.
A gluing formula we use is:

\begin{thm}[{\cite[Theorem 9.5]{BK20gluing}}]
\label{thm: BK gluing}
We have 
\[
\FSW(X,\fraks,\id_M\#f_0)
= \SW(M,\fraks_M)
\]
in $\Z/2$.
\end{thm}

The next property of families Seiberg--Witten invariants is the families adjunction inequality due to Baraglia~\cite{B202}:

\begin{thm}[{\cite[Theorem 1.2]{B202}}]\label{thm: family adjunction}
Let $X$ be an oriented smooth closed 4-manifold with $b^+(X)\geq3$.
Let $\fraks$ be a spin$^c$ structure on $X$ with $d(\fraks)=-1$.
Let $f \in \Diff(X,\fraks)$ be a diffeomorphism with $\FSW(X,\fraks,f)\neq0$.
Let $i : \Sigma \hookrightarrow X$ be a closed connected oriented smoothly embedded surface with $g(\Sigma)>0$ and with $[\Sigma]^2\geq0$.
Suppose that the embeddings $i: \Sigma \hookrightarrow X$ and $f\circ i : \Sigma \hookrightarrow X$ are smoothly isotopic to each other.
Then we have
\[
2g(\Sigma)-2 \geq |c_1(\fraks)\cdot [\Sigma]|+[\Sigma]^2.
\]
\end{thm}

\subsection{Infinite generation}
\label{subsection Infinite generation}

We give a procedure to get infinite generation of $\pi_0(\TDiff(X))$ for certain closed 4-manifolds $X$.
This is a Seiberg--Witten reformulation of Ruberman's argument \cite{Rub99} based on 1-parameter Yang--Mills theory\footnote{There is a slight difference from \cite{Rub99} in the construction. Ruberman~\cite{Rub99} used a diffeomorphism of $\CP^2\#2\overline{\CP}^2$ which is infinite order in $\pi_0(\Diff(\CP^2\#2\overline{\CP}^2))$, in place of $f_0 \in \Diff(S^2\times S^2)$ in our construction. The fact that one can use this simpler diffeomorphism $f_0$ (going back to \cite{BK20gluing}) is one of the advantages of families Seiberg--Witten theory compared with families Yang--Mills theory.}.
Similar ideas have been used in \cite{auckly2023smoothly,konno2023homology,Baraglia23mapping} for specific examples arising from logarithmic transformations, but with slightly different forms from us. 
For readers' convenience, we give an explicit statement and its proof in terms of infinite generation of Torelli groups in the most general setup where this argument works.

\begin{assum}
\label{assumption for infinite generation}
Let $M$ be a simply-connected closed smooth 4-manifold with $b^+(M)\geq 2$.
Suppose that the following holds:
\begin{enumerate}
\item[(i)] There are infinitely many smooth 4-manifolds $\{M_i\}_{i=1}^\infty$ that are homeomorphic to $M$.
\item[(ii)] For every $i$, there exists an orientation-preserving diffeomorphism $\phi_i : M\#S^2\times S^2 \to M_i\#S^2\times S^2$ such that
\[
(\phi_i)^\ast : H^2(M_i\#S^2\times S^2) \to H^2(M\#S^2\times S^2)
\]
is a direct sum of isomorphisms 
\[
\phi_i^1 : H^2(M_i) \to H^2(M)
\text{\ and\ }
\phi_i^2 : H^2(S^2\times S^2) \to H^2(S^2\times S^2).
\]

\item[(iii)] For every $i$, there exists a mod 2 basic class $\fraks_i$ on $M_i$ (i.e. $\SW(M_i,\fraks_i)\neq0$ mod 2)  with $d(\fraks_i)=0$ such that the spin$^c$ structures $\phi_i^1(\fraks_i)$ on $M$ are distinct for all $i$. 
\end{enumerate}

In (iii), we identify a spin$^c$ structure with its first Chern class.
\end{assum}

Using the diffeomorphism $f_0 \in \Diff_\del(S^2\times S^2 \setminus \mathrm{Int}(D^4))$ as in \eqref{eq: ingredient fnode}, we can define a diffeomorphism
\begin{align}
\label{eq: diffeo general}
g_i = \phi_i^{-1} \circ (\id_{M_i}\#f_0) \circ \phi_i \circ (\id_M\#f_0) : M\#S^2\times S^2 \to M\#S^2\times S^2. 
\end{align}
By construction, $g_i$ act trivially on $H_2(M\#S^2\times S^2)$.
Since we supposed $\pi_1(M)=1$, it follows from \cite{Q86,P86,gabai2023pseudoisotopies} that all $g_i$ are topologically isotopic to the identity. 

Throughout this article, we will use constructions where the Assumption~\ref{assumption for infinite generation} (or conditions similar to it) are met. For convenience, we will refer to the diffeomorphisms $\phi_i$'s as the \textit{mediating diffeomorphism}. We now give the archetypal argument for linear independence in the diffeomorphism groups:

\begin{thm}
\label{thm: infinite generation general}
Let $M$ be as in \cref{assumption for infinite generation} and set $X=M\#S^2\times S^2$.
Let $g_i : X \to X$ be diffeomorphisms defined by \eqref{eq: diffeo general}.
Then $\{g_i\}_{i=1}^\infty$  
generate a $\Z^\infty$-summand of $\pi_0(\TDiff(X))_\ab$.
\end{thm}

\begin{proof}
Using the identification between spin$^c$ structures with their first Chern classes,
define a spin$^c$ structure on $X$ by $\fraks_i' := \phi_i^1(\fraks_i)\#\frakt$, where $\frakt$ is the spin structure on $S^2\times S^2$.
We claim that
\begin{align}
\label{eq: Kronecker delta}
\begin{split}
&\FSW(X,\fraks_i',g_i) \neq 0 \text{\ in\ $\Z$ } (i \geq 1),\\
&\FSW(X,\fraks_i',g_j)=0 \text{\ in\ $\Z$ } (i>j\geq1),
\end{split}
\end{align}
after passing to a subsequence of $\{(\fraks_i',g_i)\}_{i=1}^\infty$.
First, let us check that it suffices to prove this claim.
Form a homomorphism
\begin{align}
\label{eq: collected homomorphism}
\bigoplus_{i=1}^\infty \FSW(X,\fraks_i',-) : \pi_0(\TDiff(X))_\ab \to \bigoplus_{i=1}^\infty \Z
\end{align}
by collecting \eqref{eq: homomorphism TDiff} for those spin$^c$ structures $\fraks_i'$.
It follows from \eqref{eq: Kronecker delta} that the image under \eqref{eq: collected homomorphism} of the subgroup of $\pi_0(\TDiff(X))_\ab$ generated by $\{g_i\}_{i=1}^\infty$ 
is a $\Z^\infty$-summand of the target.
This proves the assertion of the \lcnamecref{thm: infinite generation general}.

Next, we prove $\FSW(X,\fraks_i',g_i) \neq 0 \text{\ in\ $\Z$}$ for all $i$, as in the claim \eqref{eq: Kronecker delta}.
Recall that $\phi_i^1(\fraks_i)$ are supposed to be distinct, and there is a finiteness result on basic classes (e.g. \cite[Theorem 5.2.4]{Mo96}).
Thus, after passing to a subsequence, we can assume that $\SW(M,\phi_i^1(\fraks_i))=0$ for all $i$.
Using the homomorphism property of $\FSW$ and the conjugation invariance, we have the following equalities over $\Z/2$:
\begin{align}
\label{eq: FSW computation}
\begin{split}
&\FSW(X,\fraks_i',g_i)\\
=& \FSW(X,\fraks_i',\phi_i^{-1} \circ (\id_{M_i}\#f_0) \circ \phi_i)
+ \FSW(X,\fraks_i',\id_M\#f_0)\\
=& \FSW(M_i\#S^2\times S^2,\fraks_i\#\frakt, \id_{M_i}\#f_0)
+ \FSW(X,\fraks_i',\id_M\#f_0)\\
=& \SW(M_i,\fraks_i)
+ \SW(M,\phi_i^1(\fraks_i))
= 1+0=1.
\end{split}
\end{align}
Here we used \cref{thm: BK gluing} to get the third equality.
This, of course, implies $\FSW(X,\fraks_i',g_i)\neq0$ in $\Z$.

Now we see the remaining part of \eqref{eq: Kronecker delta}.
(Note that we have $\FSW(X,\fraks_i',g_j)=0$ in $\Z/2$ for $i>j$ by an argument similar to the above paragraph, but not necessarily in $\Z$.)
Firstly, note that $\fraks_i'$ are distinct since $\phi_i^1(\fraks_i)$ are distinct.
Hence it follows from \cref{prop: finiteley many families basic classes} that, for a fixed $j$, we have $\FSW(X,\fraks_i',g_j)=0$ in $\Z$ for all $i \gg j$.
Therefore, after passing to a subsequence of $\{(\fraks_i',g_i)\}_i$, we can assume that $\FSW(X,\fraks_i',g_j)=0$ in $\Z$ for $i>j$.
Thus we get the remaining part of \eqref{eq: Kronecker delta}.
This completes the proof.
\end{proof}

\subsection{Families gluing along $\neq S^3$}\label{sec: family_gluing}

We also need to recall a families gluing theorem due to Lin \cite{lin2022family}, which is a different type of gluing from the one by Baraglia and the first author \cite{BK20gluing} discussed above.
This shall be used to prove the results on universal families corks.

First, recall from \cite{lin2022family} that one can define the families cobordism map for diffeomorphisms.
Here we review the minimal setup we need.
Let $W$ be a smooth compact oriented 4-manifold with boundary $\del W = Y$.
For simplicity, we assume that $Y$ is an integral homology 3-sphere.
We denote by $\mathring{W}$ the punctured $W$;  $\mathring{W} = \mathring{W} \setminus \mathrm{Int}(D^4)$.
Let $\fraks$ be a spin$^c$ structure on $W$.
Let $f : M \to M$ be an orientation-preserving diffeomorphism with $f^\ast \fraks \cong \fraks$.
Suppose that $f$ acts trivially on homology.
Then we can define a families cobordism map
\begin{align}
\label{eq: HM hat map}
\HMhat_\ast(\mathring{W}, \fraks,f) : \HMhat_\ast(S^3) \to \HMhat_{\ast}(Y, \fraks|_Y),
\end{align}
which is a map between the ``from" version of monopole Floer homologies defined by Kronheimer--Mrowka \cite{KM07thebook}.
The degree of the map \eqref{eq: HM hat map} is given by
\begin{align}
\label{eq: deg of cob map}
\frac{c_1(\fraks)^2 -\sigma(W)}{4} - \frac{\chi(\mathring{W}) +\sigma(W)}{2} + 1.
\end{align}
If one drops the last term ``$+1$", this is the formula of the grading shift for the usual cobordism maps defined in \cite{KM07thebook}. This ``$+1$" comes from the dimension of the base space of the mapping torus of $f$.

To define \eqref{eq: HM hat map}, let us isotope $f$ near $D^4 \subset W$ and make it fixing $D^4$ pointwise.
Then the mapping torus $\mathring{W} \to E_f \to S^1$ of $f|_{\mathring{W}}$ is a smooth fiber bundle with fiber $\mathring{W}$ over $S^1$, equipped with a trivialization of the boundary family $\del E_f$.
Further, it follows from $f^\ast \fraks \cong \fraks$ that there is a fiberwise spin$^c$ structure on $E_f$ that restricts to $\fraks$ on the fiber.
For such a family, we can consider the family cobordism map defined in \cite[Proposition 4.5]{lin2022family}.
Note that homological triviality assumption is made in \cite{lin2022family}, so we restrict ourselves to the case that $f$ is topologically trivial.

The definition of \eqref{eq: HM hat map} could look to depend on the choice of the isotopy discussed above.
It is not hard to see this ambiguity actually does not affect \eqref{eq: HM hat map}.
However, we omit the proof, since what we use is only the pairing formula (\cref{thm: Lin's gluing} below), and it does not matter even if \eqref{eq: HM hat map} depends on some auxiliary data.

The families gluing theorem due to Lin~\cite{lin2022family} applied to our setup is as follows.
Let $W'$ be a smooth compact oriented 4-manifold with boundary $\del W = -Y$ and with $b^+(W')\geq 2$.
Set $X = W \cup_Y W'$.
Suppose that we have a spin$^c$ structure $\fraks'$ on $W'$, such that $\fraks'|_{\del W'} \cong \fraks|_{\del W}$, and $d(\fraks_X)=-1$ for $\fraks_X := \fraks \cup \fraks'$.
Now, extend $f \in \Diff_{\del}(W)$ to a diffeomorphism $f : X \to X$ by the identity, denoted by the same notation.
Then we can define the families Seiberg--Witten invariant $\FSW(X,f,\fraks_X) \in \Z$.

\begin{thm}[{\cite[Theorem L (1)]{lin2022family}}]
\label{thm: Lin's gluing}
In the above setup, we have
\begin{align*}
\FSW(X,f,\fraks_X)
 = \langle \HMhat_\ast(\mathring{W}, \fraks,f)(\hat{1}), \HMarrow^\ast(\mathring{W}',\fraks')(\check{1}) \rangle.
\end{align*}
\end{thm}

Here $\mathring{W'}$ is regarded as a cobordism from $Y$ to $S^3$, and $\HMarrow^\ast(\mathring{W}',\fraks') : \HMcheck^\ast(S^3) \to \HMhat^\ast(Y, \fraks'|_{\del W'})$ is the standard (i.e. non-families) ``arrow" flavor cobordism map defined in \cite{KM07thebook}.
The elements $\hat{1}$ and $\check{1}$ are generators of the $U$-towers in $\HMhat_\ast(S^3)$ and in $\HMcheck^\ast(S^3)$, respectively.
The pairing is taken between Floer (co)homology $\HMhat_\ast(Y,\fraks|_{\del W})$ and $\HMhat^\ast(Y,\fraks|_{\del W})$.

\section{Proof of localization}
\label{section localization}

In this section, we will prove (i) of Theorem~\ref{thm: localization intro}. For the proof, it will be important to have a light introduction to corks. 
\begin{defi}\label{def_cork}
A cork is a pair $(C,\tau)$, consisting of a 4-contractible manifold $C$ with boundary and an involution $\tau$
\[\tau: \partial C \rightarrow \partial C 
\]
which does not extend smoothly to $C$. We will refer to the action of $\tau$ as the \textit{cork-twist}.
\end{defi}

By the work of Freedman \cite{freedman1982topology}, it is known that the cork twist does extend over $C$ as a homeomorphism. The first example of a cork was exhibited by Akbulut \cite{akbulut1991fake}, see Figure~\ref{fig:mazur_cork}. The significance of cork-twist in the study of exotic smooth structure is demonstrated by the following result by Curtis--Freedman--Hsiang--Stong \cite{curtis1996decomposition} and Matveyev \cite{matveyev1996decomposition}, that any two exotic smooth structures on a simply-connected closed 4-manifold $X$ are related by a cork-twist, i.e. by the process of re-gluing an embedded cork $C$ by cork-twist $\tau$ along the boundary:
\[
X^{\tau}:= X \setminus \mathrm{int}(C) \cup_{\tau} C
\]
Note that since $\tau$ extends over $C$ topologically, there is a homeomorphism from $X^{\tau}$ to $X$.
In this article, we will also be interested in examples of corks for which the cork-twist \textit{does} extend over $C \# S^2 \times S^2$ smoothly. Indeed, it can be seen that most of the well-known corks have this property, that the `exotic' nature of the cork-twist diffeomorphism is diffused after one stabilization by $S^2 \times S^2$. Indeed, we have the following Proposition~\ref{pro: cork_stab}. 
\begin{figure}[h!]
\center
\includegraphics[scale=0.40]{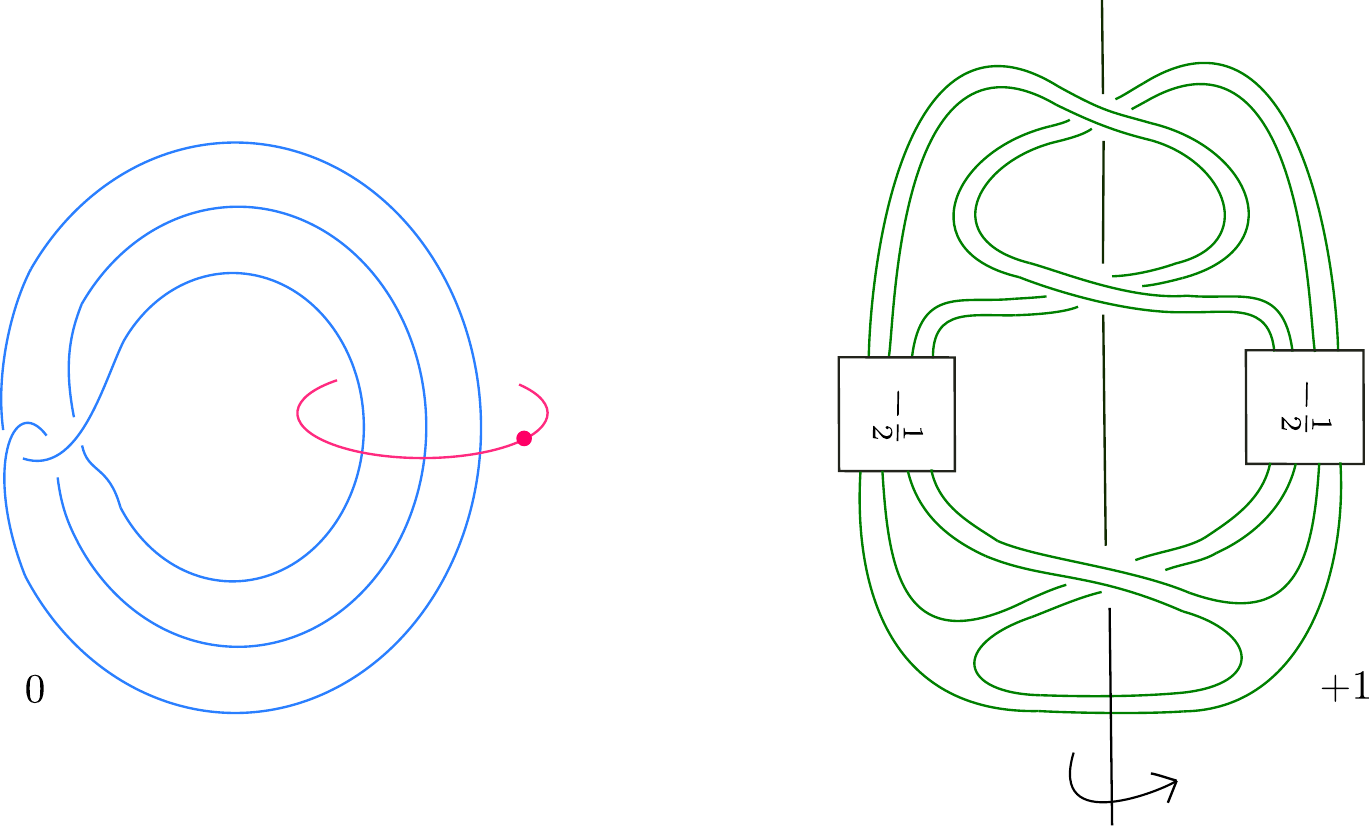}
\caption{Left: The Akbulut--Mazur cork, Right: The boundary of Akbulut--Mazur cork, as a surgery on a ribbon knot.}
\label{fig:mazur_cork}
\end{figure}

\begin{pro}\label{pro: cork_stab}
Let $(Y, \tau)$ be the  Akbulut--Mazur cork. Then $\tau$ extends as a diffeomorphism of $C \# S^2 \times S^2$. That is there is a diffeomorphism
\[
\tilde{\tau}: Y \# S^2 \times S^2 \rightarrow Y \# S^2 \times S^2.
\]
such that $\tilde{\tau}|_{\partial C}=\tau$.
\end{pro}

\begin{proof}
We use the techniques from \cite{auckly2015stable}. We first blow-up the Akbulut cork, $Y \# \mathbb{CP}^2$. It was shown in \cite[Theorem 4.2]{auckly2015stable}, there is a diffeomorphism $\rho$ from  $Y \# \mathbb{CP}^2$ to the trace of a $(+1)$-surgery along a slice knot $K$, $X_{+1}(K)$ (shown in Figure~\ref{fig:mazur_cork}). This knot is symmetric with a strong involution $\bar{\tau}$. It follows that $\bar{\tau}$ smoothly extends over the trace of the surgery \cite[Section 5]{dai2022corks}. Moreover, in the proof of \cite[Theorem 4.2]{auckly2015stable} the authors showed that $\rho$ is equivariant with respect to the symmetry on the boundaries of $Y \# \mathbb{C}P^2$ and $X_{+1}(K)$. Now it was observed in the same proof that under $\rho$ the $(+1)$-framed sphere $\mathbb{CP}^1 \subset \mathbb{CP}^2$ correspond to the $(+1)$-sphere $S \subset X_{+1}(K)$ obtained by capping of a ribbon disk $D$ of $K$ by the core of the 2-handle. One could also consider another $(+1)$-framed sphere in the trace, where we cap-off the disk $\bar{\tau}(D)$ by the core of the 2-handle, we denote this sphere by $S^{\tau}$. By construction, blowing-down $S$ gives back $C$, while blowing-down $S^{\tau}$ gives $Y$ with the twisted boundary (by $\tau$). Hence, in order to show the existence of the diffeomorphism $\tilde{\tau}: Y \# S^2 \times S^2 \rightarrow Y \# S^2 \times S^2$ which restricts to $\tau$ on $\partial Y$, it is enough to show that $S$ and $S^{\tau}$ are isotopic in the stabilized (by $S^2 \times S^2$) trace. This was shown in the proof of \cite[Theorem 4.2]{auckly2015stable}, by considering the \textit{key stable isotopy} defined by the authors.
\end{proof}

We will also need the notion of \textit{Andrew--Curtis corks} (or AC corks) defined in \cite{melvin2021higher} by Melvin--Schwartz. For the reader's convenience, we recall the definition below:
\begin{defi}
A cork is called an AC cork if it is built from the 4-ball by attaching an equal number of homotopically canceling 1 and 2-handles. In general, a 4-manifold is called an AC manifold if it is built only by attaching 1 and 2-handles to the 4-ball such that a subset of 2-handles homotopically cancels some subset of 1-handles and the remaining 2-handles are attached homotopically trivially. If we need to specify the handle structure of such AC cork $C$, we will denote it by 
\[
C= [L_1, L_2]
\] 
where $L_i$ is a link representing the set of $i$-handles. That is we think of 1-handles represented by dotted unlink and 2-handles represented by their framed attaching circles.
\end{defi}
For example, the Akbulut--Mazur cork is an AC cork. Such corks will be important for our construction. In \cite{melvin2021higher}, the authors further categorized the AC structures according to whether the 2-handles cancel all the 1-handles or not:

\begin{itemize}
\item An AC structure on a 4-manifold is said to be of type-I if X is homotopy equivalent to $\bigvee S^1$'s.

\item  An AC structure on a 4-manifold is said to be of type-II if X is homotopy equivalent to $\bigvee S^2$'s.

\end{itemize}
See Figure~\ref{fig:ac_manifold} for an example of an AC manifold.
\begin{figure}[h!]
\center
\includegraphics[scale=0.42]{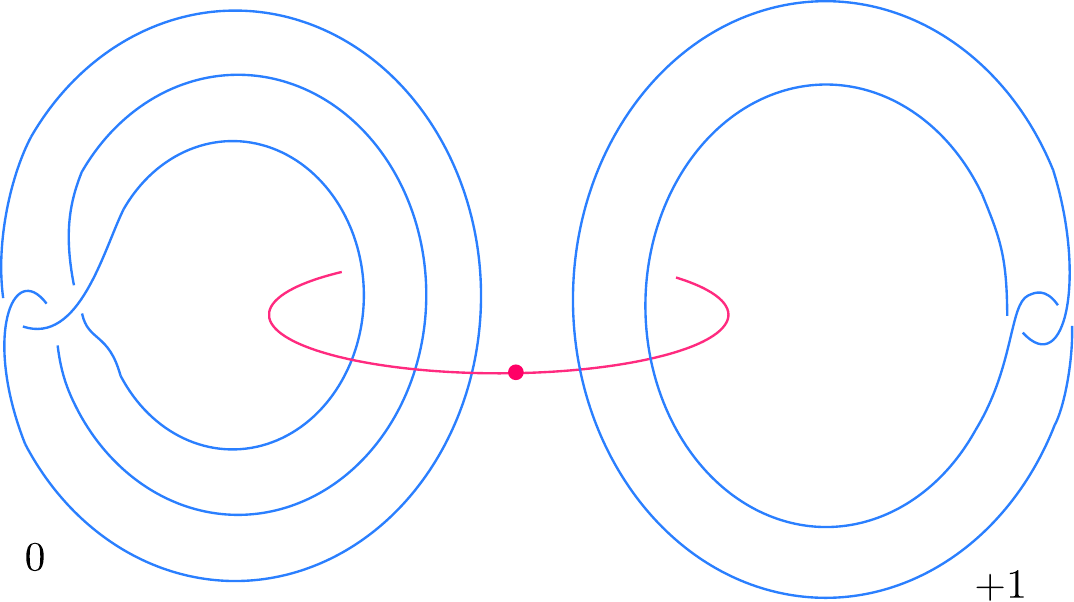}
\caption{An example of an AC manifold of Type-II}
\label{fig:ac_manifold}
\end{figure}
One key ingredient in the discussion that follows, is the work of Akbulut and Yasui \cite{akbulut2009knotting}, where the authors constructed infinitely many knotted embeddings of Akbulut cork $Y$. We will use these different embeddings in our construction. We now prove the localizability part of (i) in \cref{thm: localization intro}, which we recall for readers' convenience:

\begin{thm}
\label{thm: localization}
Let $X$ be $K3 \# \overline{\mathbb{CP}}^2$. Then there exists a $\Z^\infty$-summand $\mathcal{Z}$ of $\pi_0(\TDiff(X))_\ab$, and a contractible open submanifold $C \subset X$ such that $\mathcal{Z}$ can be localize to $C \# S^2 \times S^2 \subset X \#S^2 \times S^2$.
\end{thm}

\begin{proof}
We begin with a topological construction. Firstly, we recall the construction of Akbulut and Yasui in \cite{akbulut2009knotting}. The authors begin by considering the specific embedding of the Akbulut cork $Y$ in $X: = E(2) \# \overline{\mathbb{CP}}^2$. Then the cork $Y$ is reglued by the boundary twist $\tau$. Let us represent the twisted manifold by $X^{\tau}$ which is homeomorphic to  $E(2) \# \overline{\mathbb{CP}}^2$ and has an embedding of twisted $Y$. Note that $X$ originally had an embedding of the Gompf nucleus $N(2)$ away from the embedding of $Y$ \cite{gompf1991nuclei}. In particular, there is still an embedding of $N(2)$ in $X^{\tau}$. Consider performing the Fintushel-Stern knot surgery along a knot $K$ on the Gompf nucleus. We denote the resulting manifold by $X^{\tau}_K$. Note that the twisted $Y$ is still embedded in  $X^{\tau}_K$. As observed in \cite{akbulut2009knotting}, in $X^{\tau} \setminus N(2)$, one can split-off a $S^2 \times S^2$-summand. Hence it follows from 
\cite{akbulut2002variations, auckly2003families,BS13}
that there is a diffeomorphism
\[
\Phi_K: X^{\tau} \rightarrow X^{\tau}_K
\]
Let us define 
\[
C_K:= \Phi_K^{-1}(Y) \subset X^{\tau}
\]
Now observe that regluing $C_K$ by the boundary twist corresponds to regluing the twisted $Y \subset X^{\tau}_K$ by the boundary twist. Hence the resulting 4-manifold (after regluing $C_K$) is diffeomorphic to $X_K$. In particular, if we vary the knots $K$ among a family $\{K_p\}_{p \in \mathbb{N}}$ such that the Alexander polynomials of $K_p$ are pair-wise different, then it follows from the Fintushel--Stern knot surgery formula for Seiberg--Witten invariant \cite{fintushel_knot_surgery} that $X_{K_{p}}$'s are pair-wise not diffeomorphic (while being homeomorphic). This shows that the corresponding family of embeddings of $Y$ in $X^{\tau}$, given by $\{C_{K_{p}}\}$ are not smoothly isotopic.

Let us now consider $K_i:=T_{2,2i+1}$, for any $i \in \{1,2, \cdots, n \}$, for any $n$. We now claim that the exists family of contractible 4-manifolds $\{\overline{C}_n\}$ satisfying the following property
\[
\bigcup\limits_{i=1}^{n}C_{K_i} \subset \overline{C}_n, \; \; \mathrm{and} \; \; \overline{C}_{n} \subset \overline{C}_{n+1}.
\]
To build this $\overline{C}_n$, we argue similarly as \cite{melvin2021higher}. Firstly, observe that each of $C_{K_i}$ is an AC cork, since they are different embeddings of the Akbulut--Mazur cork. Define $\overline{C}_1:=C_{K_1}$. By an ambient isotopy, we pull apart the 1-handles of $C_{K_1}$ and $C_{K_2}$ so that they lie in separate disjoint 3-balls and the core of the 2-handles intersect transversely. Following the proof of Finger Lemma from \cite[Lemma 2.4]{melvin2021higher}, we now form the `union' of $C_{K_1}$ and $C_{K_2}$ by plumbing the 2-handles together according to the intersections of their cores, see Figure~\ref{fig:union_cork}.
\begin{figure}[h!]
\center
\includegraphics[scale=0.40]{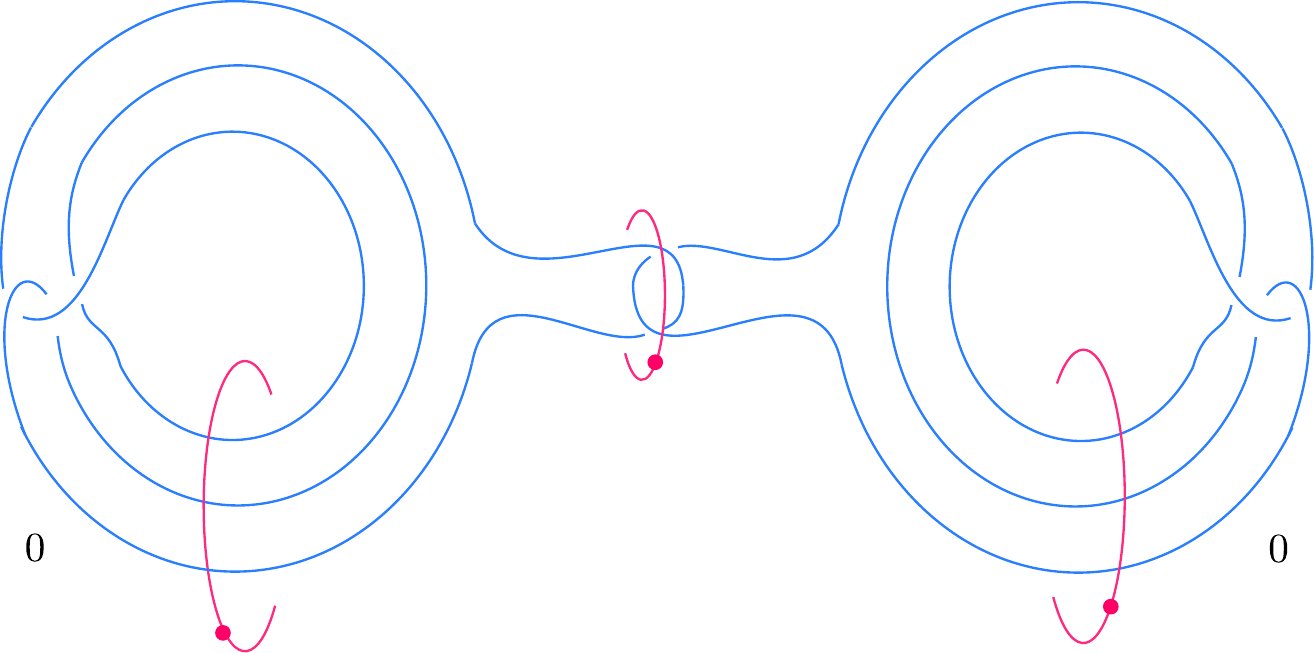}
\caption{An AC manifold of type-I formed by the \textit{union} of the two corks, note that the number of clasps in the middle depends on the intersections of the core of the two handles in the embedding.}
\label{fig:union_cork}
\end{figure}
Note that in this union of $C_{K_1}$ and $C_{K_2}$ is an AC manifold of type-I. Then we apply Encasement Lemma from \cite[Lemma 2.2]{melvin2021higher} to get an embedded AC 4-manifold $\overline{C}_{2}$ of type-II, such that $H_2(\overline{C}_2)=H_2( C_{K_1} \cup C_{K_2})=0$. Hence $\overline{C}_{2}$ is contractible. By construction:
\[
C_{K_1} \cup C_{K_2} \subset \overline{C}_{2} \; \; \mathrm{and} \; \; \overline{C}_{1} \subset \overline{C}_{2}.
\]
Now since $\overline{C}_2$ is an AC cork, we repeat the same argument with $\overline{C}_{2}$ and $C_{K_3}$ (in place of $C_{K_1}$ and $C_{K_2}$) from above to get another AC cork $\overline{C}_{K_3}$. Again, by construction:
\[
\bigcup\limits_{i=1}^{3}C_{K_i} \subset \overline{C}_{K_3} \; \; \mathrm{and \; \;} \overline{C}_2 \subset \overline{C}_3.
\]
We proceed inductively to get $\{\overline{C}_{n}\}$. Let us connected sum $S^2 \times S^2$ to $X^{\tau}$ by deleting a $3$-ball from $\overline{C}_{n}$. We now define diffeomorphisms
\[
\overline{\Phi}_{K_i}: X^{\tau} \# S^2 \times S^2 \rightarrow X^{\tau}_{K_i} \# S^2 \times S^2.
\]
by first isotoping the $S^2 \times S^2$ factor\footnote{Note that this factor is unrelated to the internal $S^2 \times S^2$-factor used to define $\Phi_K$ above.} (by finger moves) in $\overline{C}_n$ so that it is connected to $C_{K_i}$, then post-composing with
\[
\Phi_{K_i} \# \mathrm{id} : X^{\tau} \# S^2 \times S^2 \rightarrow X^{\tau}_{K_i} \# S^2 \times S^2.
\]
This diffeomorphism $\overline{\Phi}_{K_i}$ maps the stabilized contractible manifold $\overline{C}_n \# S^2 \times S^2$ to a stabilized  contractible manifold $\overline{\Phi}_{K_i}(\overline{C}_{n}) \# S^2 \times S^2 \subset X^{\tau}_{K_i} \# S^2 \times S^2$. Moreover, note that since the finger-move isotopy does not move $C_{K_i}$, by definition of $C_{K_i}$, $Y \# S^2 \times S^2 \subset \overline{\Phi}_{K_i}(\overline{C}_{n}) \# S^2 \times S^2$. In particular, it follows from Proposition~\ref{pro: cork_stab} that there exist a diffeomorphism 
\[
\tilde{\tau_{i}}: X^{\tau}_{K_i} \# S^2 \times S^2 \rightarrow X_{K_i} \# S^2 \times S^2
\]
Note since this diffeomorphism is supported in $Y \# S^2 \times S^2$, it follows that $\tilde{\tau_{i}}(\overline{\Phi}_{K_i}(\overline{C}_{n}) \# S^2 \times S^2)$ is again a once stabilized contractible manifold $\tilde{\tau_{i}}(\overline{\Phi}_{K_i}(\overline{C}_{n})) \# S^2 \times S^2$. In other words, $\tilde{\tau_{i}}|_{S^2 \times S^2} \subset \tilde{\tau_{i}}(\overline{\Phi}_{K_i}(\overline{C}_{n})) \# S^2 \times S^2$. Let us now consider the composition
\[
\phi_i:= \tilde{\tau_{i}} \circ \overline{\Phi}_{K_{i}} : X^\tau \# S^2 \times S^2 \rightarrow X_{K_i} \# S^2 \times S^2.
\]
which sends a once stabilized contractible manifold $\overline{C}_n \# S^2 \times S^2$ to a once stabilized contractible manifold $\tilde{\tau_{i}}(\overline{\Phi}_{K_i}(\overline{C}_{n})) \# S^2 \times S^2$.

Following the preceding section, we will use $\phi_i$ as the mediating diffeomorphism to define a sequence of diffeomorphism $g_i$ of $X^{\tau} \# S^2 \times S^2$. 

Firstly, let us consider a diffeomorphism of $S^2 \times S^2$, given by 
\begin{align*}
f_0: S^2 \times S^2 \rightarrow S^2 \times S^2 \\
(x,y) \mapsto (r(x),r(y))
\end{align*}
where $r:S^2 \rightarrow S^2$ is the reflection along the equator $S^1 \subset S^2$. Let us assume that we isotope $f_0$ so that it fixes a small 4-disk in $S^2\times S^2$. Hence, we can regard $f_0 \in \Diff_{\del}(S^2\times S^2 \setminus \mathrm{Int}(D^4))$, hence as a diffeomorphism of $W \# S^2 \times S^2$ by extending it by the identity, for any closed 4-manifold $W$\footnote{We will denote this diffeomorphism as $f_0$ for every $W$ if there is no risk of confusion.}. Let us now consider the composition $g_i$:
\begin{align}
\label{eq: gi}
\begin{split}
g_i :& X^{\tau} \# S^2 \times S^2 \xrightarrow{\mathrm{id} \# f_0} X^{\tau} \# S^2 \times S^2 \\
& \xrightarrow{\phi_i} X_{K_i} \# S^2 \times S^2 \xrightarrow{\mathrm{id} \# f_0}  X_{K_i} \# S^2 \times S^2 \xrightarrow{\phi^{-1}_i} X^{\tau} \# S^2 \times S^2.
\end{split}
\end{align}
We now argue that there is a $\mathbb{Z}^{n}$-summand in $\pi_0(\TDiff(X^{\tau} \# S^2 \times S^2))_\ab$ that can localize to $\pi_0(\TDiff_{\partial}(C_n \# S^2 \times S^2))$ for any $n$. This follows from the proof of Theorem~\ref{thm: infinite generation general}. In fact, since we will use the diffeomorphisms $g_i$ to establish this, we only need to check the existence of the  $\mathbb{Z}^{n}$-summand in $\pi_0(\TDiff(X^{\tau} \# S^2 \times S^2))_\ab$ since by construction these $g_i$ localize to $C_n \# S^2 \times S^2$ (rel $\partial$). Firstly, note that by construction there is a mod 2 basic class $\fraks_i'$ in $X_{K_i}$, such that $\fraks_i'$, when thought of as an element in $H_2(X^{\tau})$ are all distinct. Moreover, since $X^{\tau}$ splits-off an internal $S^2 \times S^2$ factor, it follows that $\SW(X^{\tau}, \fraks)=0$ for any $\spinc$-structure $\fraks$ on $X^{\tau}$. Again a similar analysis as in the preceding section yields
\begin{align*}
\FSW(X^{\tau} \# S^2 \times S^2,\fraks_i',g_i)=1
\end{align*}
over $\Z/2$, just as in \eqref{eq: FSW computation}.
Hence we get homomorphisms
\begin{align}
\label{eq: collected homomorphism corks}
\bigoplus_{i=1}^n \FSW(X,\fraks_i',-) : \pi_0(\TDiff(X))_\ab \to \bigoplus_{i=1}^n \Z
\end{align} 
which implies the desired claim. Let us now consider:
\[
C:=\bigcup^{\infty}_{n=1} \mathrm{Int}(C_n).
\]
Since by construction $\{C_n\}$ are nested and each $C_n$ is contractible, it follows that $C$ is a contractible, non-compact manifold. 
By the argument above there is a $\mathbb{Z}^{\infty}$-summand in $\pi_0(\TDiff_{\partial}(C \# S^2 \times S^2))_\ab$ localizing the $\mathbb{Z}^{\infty}$-summand in $\pi_0(\TDiff(X^{\tau} \# S^2 \times S^2))_\ab$. Finally note that although we prove our result for $X^{\tau} \# S^2 \times S^2$, which is diffeomorphic to $X \# S^2 \times S^2$, hence the claimed statement follows.
\end{proof}

\section{Proof of non-localization}
\label{section Proof of non-localizability}

Now we prove our non-localization results.
Our argument is inspired by Auckly's work \cite{auckly2023smoothly}.
Before starting the proof, let us recall the following \lcnamecref{lem: geography}, which is based on a geography result by Park \cite{Park02}.

\begin{lem}[{\cite[Lemma  4.4]{auckly2023smoothly}}]
\label{lem: geography}
Let $X$ be a simply-connected closed smooth 4-manifold. 
Then there are $N \geq 0$ and symplectic 4-manifolds $M_X$, $M_{-X}$ with the following properties: 
\begin{itemize}
    \item[(i)] both $M_X$ and $M_{-X}$ contain $N(2)$,
    \item[(ii)] $X\#N(S^2\times S^2)$ is orientation-preservingly diffeomorphic to $M_X\#S^2\times S^2$, and orientation-reversingly diffeomorphic to $M_{-X}\#S^2\times S^2$.
\end{itemize}
\end{lem}

Now we prove our main non-localization result:

\begin{proof}[Proof of Theorem~\ref{thm: intro stabilized 4-manifolds}]
We will first consider the case when $\alpha^2 \geq 0$ with $[\Sigma]= \alpha$. Our argument is quite inspired by that in \cite[Section 4]{auckly2023smoothly}. Suppose that $X$ is spin and $\alpha$ is characteristic. Consider $M_X$ from Lemma~\ref{lem: geography}. Let $\sigma$ and $\delta$ respectively denote the section and fiber class of the nucleus $N(2) \subset M_X$. Then it follows from \cite[Lemma 4.1]{auckly2023smoothly} and the work of Wall \cite{Wa64} that there is a diffeomorphism 
\[
\Phi : M_X \# S^2 \times S^2 \to M_X \# S^2 \times S^2
\]
such that the action of $\Phi_*$ on $H_2(M_X \# S^2 \times S^2)$ sends $\alpha$ to $[2A_1 (\sigma+ \delta) + 2A_1 A_2 \delta]$, where $A_i$ are integers with $A_1 > 0$ and $A_2 \geq 0$.

As usual, we identify a spin$^c$ structure on a simply-connected 4-manifold with its first Chern class.
Let $K$ be the canonical basic class of $M_X$.
Since both $\sigma + \delta$ and $\delta$ are represented by torus and $K$ is a basic class of $M_X$, the (ordinary) adjunction inequality applied to $(M_X, K, T^2)$ implies
\begin{align}
\label{eq: ordinary adjunction}
K\cdot(\sigma + \delta)=0, \; \; \mathrm{and} \; \; K\cdot\delta=0.
\end{align}
The adjunction inequality implies also that 
\begin{align}
\label{eq: vanishig thanks to adj}
\SW(M_X, K + 2i \delta) = 0 \quad(i>0).
\end{align}

Now as in \eqref{eq: diffeo general}, we define a sequence of diffeomorphisms $\{ g_i \}$ on $M_X \# S^2 \times S^2$. To specify these diffeomorphisms, it is enough to define the mediating diffeomorphisms. In this case, the mediating diffeomorphism $\phi_i$ will be given by log-transformation on the nucleus:
\[
\phi_i:M_X \# S^2 \times S^2 \rightarrow M_i \# S^2 \times S^2.
\]
Here $M_i = (M_X)_{2i+1}$ represents the closed 4-manifold obtained from $M_X$ by performing the $(2i+1)$-log transformation on the embedded distinguished nucleus. Again it follows from \cite{Gom91}
that such diffeomorphisms $\phi_i$ exists and satisfies the conditions from Assumption~\ref{assumption for infinite generation}.
By Wall's theorem \cite{Wa64}, we can
arrange $\phi_i$ so that $\phi_i^\ast$ decomposes into $\phi_i^\ast = \phi_i^1 \oplus \phi_i^2$ as in \cref{assumption for infinite generation}, where $\phi_i^1 : H^2(M_i) \to H^2(M_X)$ and $\phi_i^2 : H^2(S^2\times S^2) \to H^2(S^2\times S^2)$ are isomorphisms, and $\phi_i^1$ can be assumed to be obtained from an isomorphism (denoted by the same notation) $\phi_i^1 : H^2(N(2)_{2i+1}) \to H^2(N(2))$ and by extending it by the identity.

It remains to define the $\spinc$-structures $\mathfrak{s}_i$ to ensure all conditions of Assumption~\ref{assumption for infinite generation} are met.  
Define:
$\mathfrak{s}_i:= (\phi_i^1)^{-1}(K + 2i \delta).$
By a surgery formula on Seiberg--Witten invariants \cite{MMS97}, for $j \in \Z$, $(\phi_i^1)^{-1}(K+2j\delta)$ are mod 2 basic class of $M_i$ exactly when $|j| \leq i$.
Thus $\fraks_i$ is a mod 2 basic class of $M_i$.
Now it is clear that all conditions of Assumption~\ref{assumption for infinite generation} are met.
We define a spin$^c$ structure  $\mathfrak{s}_i'$ on $X'$ by
\[
\mathfrak{s}_i' = \phi_i^1(\mathfrak{s}_i)\#\frakt = K + 2i \delta,
\]
where $\mathfrak{t}$ be the spin structure of $S^2 \times S^2$.

Thus it follows from Theorem~\ref{thm: infinite generation general} that there is a $\mathbb{Z}^{\infty}$-summand of $\pi_0(\TDiff(X'))_{\ab}$ generated by $\{ g_i \}$.
Set
\[
f_i= \Phi^{-1} \circ g_i \circ \Phi : X' \to X'
\]
Then $\{f_i\}$ also generate a $\Z^\infty$-summand,
which we denote by $\mathcal{Z}(\alpha)$.
For a multi-index 
\[
\nu = (n_1,n_2,\dots) \in (\Z_{\geq0})^{\oplus\N},
\]
we get a well-defined diffeomorphism
\[
f^\nu := \prod_{i=1}^\infty f_{i}^{n_i} \in \TDiff(X'),
\]
since $f_i^{n_i}=\id_X$ except for finitely many $i$.
We look at the image of the surface $\Sigma$ under this diffeomorphism:
\[
\Sigma^{\nu} := f^\nu(\Sigma).
\]

Now assume that
\begin{align}
\label{eq: assumption for argument by contradiction}
 [f^\nu] \in \im(\pi_0(\TDiff_c(X'\setminus \Sigma))_\ab \to \pi_0(\TDiff(X'))_\ab).    
\end{align}
It follows that $\Sigma^\nu$ is smoothly isotopic to $\Sigma$.
Let $i(\nu) \geq 0$ denote the largest $i$ with $n_i\neq0$.
We shall see that we have a contradiction for $\nu$ with sufficiently large $i(\nu)$.

To see this, define
\[
g^\nu := \prod_{i=1}^\infty g_{i}^{n_i} \in \TDiff(X'),
\]
and set
\[
\Sigma' = \Phi(\Sigma),\quad (\Sigma')^\nu = g^\nu(\Sigma').
\]
Note that $\Sigma^\nu$ is isotooic to $\Sigma$ if and only if $(\Sigma')^\nu$ is isotopic to $\Sigma'$. Thus, by the current assumption, we have $(\Sigma')^\nu$ is isotopic to $\Sigma'$.

Now we claim that
\begin{align}
\label{eq: vanishing SW ij}
\SW(M_i,\phi_j^\ast\fraks_i')=0\quad (i>j).
\end{align}
To see this, recall that $\phi_i^1$ is obtained from an isomorphism (denoted by the same notation) $\phi_i^1 : H^2(N(2)_{2i+1}) \to H^2(N(2))$.
Note that, for $x\in H^2(N(2))$, we have $x \cdot \delta=0$ if and only if $x \in \Z\cdot\delta$.
Thus, it follows from \eqref{eq: ordinary adjunction} that $K|_{N(2)} \in \Z \cdot \delta$.
Hence $(K+2i\delta)|_{N(2)}$ is a multiple of $\delta$.
Set $\delta'_i := (K+2i\delta)|_{N(2)} \in H^2(N(2))$.
Since the map assigning $i$ the divisibility of $\delta'_i$ is injective, we have \eqref{eq: vanishing SW ij} by the description of the set of mod 2 basic classes of $M_i$ mentioned above.

Now a computation analogous to \eqref{eq: FSW computation} implies that 
\begin{align}
&\FSW(X',\fraks_i', g_i)\neq0 \quad (i\geq1), \label{eq: ii}\\
&\FSW(X',\fraks_i', g_j) = 0 \quad (i>j\geq0).
\label{eq: ij}
\end{align}
Here, to get \eqref{eq: ii} without passing to a subsequence, we have used \eqref{eq: vanishig thanks to adj}.
To get \eqref{eq: ij}, we have used \eqref{eq: vanishing SW ij}.

It follows from \eqref{eq: ii}, \eqref{eq: ij} that 
\begin{align*}
\FSW(X',\fraks_{i(\nu)}', g^\nu)
=& \sum_{i=1}^{i(\nu)} n_i\FSW(X',\fraks_{i(\nu)}', g_i)\\
=& n_{i(\nu)}\FSW(X',\fraks_{i(\nu)}', g_{i(\nu)})
\neq0.
\end{align*}

Now we apply the family adjunction inequality from Theorem~\ref{thm: family adjunction} to $(X',\fraks_{i(\nu)}', g^\nu)$ with the surface $\Sigma'$ to obtain the following:
\begin{align*}
\begin{split}
2g(\Sigma) -2 
&\geq 
|(K+2i(\nu) \delta)\cdot\Phi_\ast(\alpha)| + \Phi_\ast(\alpha)\cdot\Phi_\ast(\alpha) \\
&= |(K + 2i(\nu) \delta)\cdot(2A_1 (\sigma+\delta) + 2A_1A_2 \delta)| + (2A_1 (\sigma+\delta) + 2A_1A_2 \delta)^2 \\
&= 8c^{2}_{1}A_2 + 4A_1 i(\nu).
\end{split}
\end{align*}
Here in this computation, in the third line, we have used the following:
Firstly, we have used \eqref{eq: ordinary adjunction}.
Secondary,
since $\sigma$ and $\delta$ are a basis of an hyperbolic summand of the intersection form, we get
\[
(\sigma+\delta)^2=\delta^2=0, \; \; \mathrm{and} \; \; (\sigma+\delta)\cdot\delta=1.
\]

Thus
\[
\calZ(\alpha) \cap \im(\pi_0(\TDiff_c(X'\setminus \Sigma))_\ab \to \pi_0(\TDiff(X'))_\ab)
\]
is generated by 
\[
\Set{[f_i] | 2g(\Sigma)-2 \geq 8A_1^2A_2+4A_1i},
\]
which is a finite set, as $A_1 > 0$.

How to modify the above proof for other variations of $X$ (i.e. non-spin $X$) and $\alpha$ (i.e. non-characteristic $\alpha$ or $\alpha^2 < 0$) is totally analogous to arguments in the proof of \cite[Theorem 1.5]{auckly2023smoothly}, so leave it to the reader.
We just remark that the 4-manifold $M_{-X}$ in \cref{lem: geography} is used to treat the case that $\alpha^2<0$.

Lastly, we check the ``furthermore part" of the statement.
Given $n\neq0$, let $\Sigma$ be a surface that represents $n\alpha$, and let $\nu \in (\Z_{\geq0})^{\oplus\N}$, and assume that
\begin{align*}
 [f^\nu] \in \im(\pi_0(\TDiff_c(X'\setminus \Sigma))_\ab \to \pi_0(\TDiff(X'))_\ab).    
\end{align*}
Note that $\Phi_\ast$ takes $n\alpha$ to $n[2A_1 (\sigma+ \delta) + 2A_1 A_2 \delta]$.
By repeating the above argument, we have
 \begin{align*}
\begin{split}
2g(\Sigma) -2 
&\geq |n||(K + 2i(\nu) \delta)\cdot(2A_1 (\sigma+\delta) + 2A_1A_2 \delta)| + n^2(2A_1 (\sigma+\delta) + 2A_1A_2 \delta)^2 \\
&= 8n^2A_1^{2}A_2 + 4|n|A_1 i(\nu).
\end{split}
\end{align*}
Thus, if we put $\calZ(n\alpha)=\calZ(\alpha)$,
\[
\calZ(n\alpha) \cap \im(\pi_0(\TDiff_c(X'\setminus \Sigma))_\ab \to \pi_0(\TDiff(X'))_\ab)
\]
is generated by 
\[
\Set{[f_i] | 2g(\Sigma)-2 \geq 8n^2A_1^2A_2+4|n|A_1i},
\]
which is again a finite set.
This completes the proof of the ``furthermore part".
\end{proof}

\begin{proof}[Proof of \cref{thm: intro stabilized 4-manifolds non-negative}]
The proof is totally analogous to that of Theorem~\ref{thm: intro stabilized 4-manifolds}.
The differences are only as follows.

First, we assumed that $X$ contains $N(2)$ and admits a mod 2 basic class.
Thus we do not have to use the geography result (\cref{lem: geography}): we can use $X$ in place of $M_X$ in the proof of Theorem~\ref{thm: intro stabilized 4-manifolds}.
Second, since we do not have the counterpart of $M_{-X}$ now, we cannot handle homology classes $\alpha$ with $\alpha^2<0$.
\end{proof}

Now we give proof of the results on the non-localization to small regions of 4-manifolds:

\begin{proof}[Proof of \cref{cor: intro stabilized 4-manifolds homology disk}]
By replacing $X$ with $X\#S^2\times S^2$ if necessary, we can assume that $H_2(X')\neq0$.
Pick a non-zero homology class $\alpha \in H_2(X';\Z)$ and set $\calZ = \calZ(\alpha)$.
Let $C$ be a topological rational homology 4-disk $C$ local-flatly embedded in $X'$.
Note that $\del C$ has a collar neighborhood because of local flatness \cite{Brownlocally62}.
Using this, we can get the Mayer--Vietoris sequence corresponding to the decomposition $X' = C \cup_{\del C} (X' \setminus C)$.
Noting that $\del C$ is a rational homology 3-sphere,
we have from this Mayer--Vietoris sequence that there is $\beta \in H_2(X' \setminus C;\Z)$ that hits $n\alpha$ for some $n\neq0$ via the natural map $H_2(X' \setminus C;\Z) \to H_2(X';\Z)$.

It follows from \cref{lem: surface representative} below applied to $\beta$ that there is a closed and smoothly embedded oriented connected surface representative $\Sigma$ of $n\alpha$ that lies in the open submanifold $X' \setminus C$ of $X'$.
Thus the finite generation of \eqref{eq: Z cap intro} together with $\calZ(\alpha)=\calZ(n\alpha)$ implies that any infinite-rank subgroup of $\calZ$ does not localize to $C$. 
This completes the proof.
\end{proof}

For readers' convenience, we include a proof of an extension of a well-known fact for closed 4-manifolds to non-compact 4-manifolds:

\begin{lem}
\label{lem: surface representative}
Let $Z$ be a non-compact oriented smooth 4-manifold without boundary.
For every homology class $\gamma \in H_2(Z;\Z)$, there is a closed, smoothly embedded, oriented, and connected surface $\Sigma$ that represents $\gamma$.
\end{lem}

\begin{proof}
Recall that the Poincar{\'e} duality holds also for non-compact manifolds by working with compactly supported cohomology $H_c^\ast(-)$.
Let $L \to Z$ be the complex line bundle that corresponds to the image of $\gamma$ under the map
\[
H_2(Z;\Z) \cong H_c^2(Z;\Z)
\to H^2(Z;\Z).
\]
Since this map factors through compactly supported cohomology, $L$ is trivial over $Z \setminus K$ for some compact subset $K$ of $Z$.
Thus we can take a generic smooth section $s : Z \to L$ that is nowhere vanishing over $Z \setminus K$.
Then the zero set $\Sigma := s^{-1}(0)$ gives a closed smooth surface in $Z$, which can be modified to be connected.
This $\Sigma$ is the desired surface.   
\end{proof}

\begin{proof}[Proof of \cref{cor: intro Dehn twists}]
The proof is analogous to that of  \cref{cor: intro stabilized 4-manifolds homology disk}:
We may assume $H_2(X')\neq0$, and pick $\alpha \in H_2(X';\Z) \setminus \{0\}$ and set $\calZ = \calZ(\alpha)$. 
For any $W$, we can find a closed surface $\Sigma$ with $\Sigma \cap W = \emptyset$ and with $[\Sigma]=n\alpha$ for some $n\neq0$ .
Then finite generation of \eqref{eq: Z cap intro} together with $\calZ(\alpha)=\calZ(n\alpha)$ implies finite generation of 
\begin{align}
\label{eq: cylinder}
\calZ \cap \im(\pi_0(\TDiff_\del(W))_\ab \to \pi_0(\TDiff(X'))_\ab)
\end{align}
This completes the proof.
\end{proof}

\begin{proof}[Proof of \cref{cor: intro stabilized 4-manifolds non-negative}]
This is proven by repeating the proofs of \cref{cor: intro stabilized 4-manifolds homology disk,cor: intro Dehn twists} for non-zero $\alpha$ with $\alpha^2>0$.
\end{proof}

Next, we prove the families analog of Yasui's results:

\begin{proof}[Proof of \cref{thm: universal non-exitence general}]
Take a simply-connected closed smooth 4-manifold $X$ with $b_2(X) \geq m+n$.
Let $N\geq0$ be the natural number given in \cref{thm: intro stabilized 4-manifolds}, and consider $X' = X\# N S^2\times S^2$.
We shall prove that $X'$ satisfies the desired property: for any compact codimenion-0 topological submanifold $W$ of $X'$ with $b_2(W) < m$ and $b_1(\del W)<n$, $\pi_0(\TDiff(X'))$ do not localize to $W$.

First, using the Mayer--Vietoris sequence for the decomposition $X = W \cup (X\setminus W)$, 
we deduce from 
\[
b_2(W)<m,\ b_1(\del W)<n,\ b_2(X') \geq m+n
\]
that 
\[
\im(H_2(X' \setminus W;\Z) \to H_2(X';\Z))\neq 0.
\]
Let $\alpha \in H_2(X';\Z)$ be a non-zero homology class in this image.
Set $\calZ(W)=\calZ(\alpha)$, where $\calZ(\alpha)$ is the $\Z^\infty$-summand  given in \cref{thm: intro stabilized 4-manifolds}.

As in the proof of \cref{cor: intro stabilized 4-manifolds homology disk}, we can take a closed, smoothly embedded,  oriented, and connected surface representative $\Sigma$ of $\alpha$ that lies in $X' \setminus W$.
Then finite generation of \eqref{eq: Z cap intro} implies that, any infinite-rank subgroup of $\calZ(W)$ does not localize to $W$.
\end{proof}

\begin{proof}[Proof of \cref{cor: universal non-exitence}]
We prove by contradiction:
assume that there exists such $W$.
Set $m=b_2(W)-1, n=b_1(\del W)-1$, and let $X$ be the simply-connected closed smooth 4-manifold $X$ given in \cref{thm: universal non-exitence general}.
The assumption on $W$ implies that there is an embedding $W \hookrightarrow X$ along which $\pi_0(\TDiff(X))$ localizes to $W$, but this contradicts \cref{thm: universal non-exitence general}.
\end{proof}

Now we can complete the proof of \cref{thm: localization intro} together with \cref{thm: localization}:

\begin{proof}[Proof of (ii) of \cref{thm: localization intro}]
Let $\calZ$ be 
the $\Z^\infty$-summand of $\pi_0(\TDiff(X'))_\ab$ constructed in the proof of \cref{thm: localization}.
We use the notation in the proof of \cref{thm: intro stabilized 4-manifolds}.
Set $\alpha = 2(\sigma + \delta)+2\delta \in H_2(N(2)) \subset H_2(K3)$.
By taking $\Phi=\id_{K3\#\overline{\CP}^2\#S^2\times S^2}$ and $A_1=A_2=1$, we can repeat the proof of \cref{thm: intro stabilized 4-manifolds} for $\alpha$, only with replacing log transform with knot surgery, so that we have
\[
\calZ(\alpha) \cap \im(\pi_0(\TDiff_c(X'\setminus \Sigma))_\ab \to \pi_0(\TDiff(X'))_\ab)
\]
is finitely generated for every surface $\Sigma$ that represents $\alpha$.
Note further that $\calZ=\calZ(\alpha)$ by $\Phi=\id$.

Now, note that $C$ was constructed in \cref{thm: localization} as the union $C=\bigcup_{n=1}^\infty \mathrm{Int}(C_n)$, where $C_n$ is compact and contractible.
For large enough $n$, we have $K \subset C_n\#S^2\times S^2$.
From a similar Mayer--Vietoris argument in the proof of \cref{cor: intro stabilized 4-manifolds homology disk}, it follows that there is a surface $\Sigma$ in $K3\#\overline{\CP}^2$that represents $\alpha$ and does not intersect with $C_n$, and hence with $K$.
Applying the observation of the above paragraph to this $\Sigma$, we have that any infinite-rank subgroup of $\calZ$ does not localize to $K$.
\end{proof}

By a tweak of the arguments for $\TDiff(X)$ discussed until here, we can give the proof of our result on $\Diff(X)$.

\begin{proof}[Proof of \cref{thm: localization intro Diff}]
Let $C$ be the open submanifold of $X=K3 \# \overline{\mathbb{CP}}^2$ given in \cref{thm: localization}.
Let $\calZ$ be the image of the $\Z^\infty$-summand of $\pi_0(\TDiff_c(C'))_\ab$ given in \cref{thm: localization} under the map
\begin{align*}
\pi_0(\TDiff_c(C'))_\ab
\to \pi_0(\Diff_c(C'))_\ab
\end{align*}
We shall see that this $\calZ$ satisfies the required property.

First, let us check $\calZ$ is isomorphic to $(\Z/2)^\infty$.
To see this, it suffices to see that $g_i$ constructed in \eqref{eq: gi} is of order 2 in $\pi_0(\Diff(X'))_\ab$, where $X'=X\#S^2\times S^2$.
To see this, by construction of $g_i$, it is enough to check that $\id_{X^\tau}\#f_0 : X^\tau\#S^2\times S^2 \to X^\tau\#S^2\times S^2$ and $\id_{X_{K_i}}\#f_0 : X_{K_i} \#S^2\times S^2 \to X_{K_i} \#S^2\times S^2$ are of order 2 in $\pi_0(\Diff(X^\tau\#S^2\times S^2))_\ab$ and in $\pi_0(\Diff(X_{K_i}\#S^2\times S^2))_\ab$, respectively.
This follows from that $f_0$ is of order 2 in $\pi_0(\Diff_{\del}(S^2\times S^2 \setminus \mathrm{Int}(D^4)))$.
This completes the proof that $\calZ$ is isomorphic to $(\Z/2)^\infty$.

Now, given a spin$^c$ structure $\fraks$ on $X'$ with $d(\fraks)=-1$ such that $\fraks|_{C'}$ is the unique spin structure on $C'$, we obtain a homomorphism
\[
\FSW(X', \fraks,-) : \pi_0(\Diff_c(C'))_{\ab} \to \Z/2
\]
by composing \eqref{eq: homomorphism Diff sO mod 2} with the map 
$\pi_0(\Diff_c(C'))
\to \pi_0(\Diff(X',\fraks))$ induced from the inclusion.
Collecting this, we get a homomorphism
\begin{align}
\label{eq: mod 2 collected homomor}
\bigoplus_{\fraks}
\FSW(X', \fraks,-) : \pi_0(\Diff_c(C'))_\ab \to \bigoplus_{\fraks} \Z/2.
\end{align}
The proof of \cref{thm: infinite generation general} using this homomorphism in place of \eqref{eq: collected homomorphism} shows that $\calZ$ is sent under \eqref{eq: mod 2 collected homomor} to a $(\Z/2)^\infty$-subgroup $\calZ'$ of the target of \eqref{eq: mod 2 collected homomor}.
Since $\calZ$ is isomorphic to $(\Z/2)^\infty$, we can get a section of \eqref{eq: mod 2 collected homomor} over $\calZ'$ with image $\calZ$.
Thus $\calZ$ is a $(\Z/2)^\infty$-summand of $\pi_0(\Diff_c(C'))_\ab$.

Now, by repeating proof of (ii) of \cref{thm: localization intro}, we have that any infinitely generated subgroup of the $(\Z/2)^\infty$-summand $\calZ$ does not localize to any compact $K$.
\end{proof}

\section{Localization obstructions and family gluing}
\label{senction Proof of the results on universal families corks}
This section contains the proof of various non-localization results that use family gluing theorems, see Subsection~\ref{sec: family_gluing}. We begin by recording a pair of obstructions that will be useful.

\begin{lem}
\label{thm: general vinishing FSW}
Let $X$ be a closed smooth oriented 4-manifold with $b^+(X)\geq3$, and $\fraks$ be a spin$^c$ structure on $X$ with $d(\fraks)=-1$.
Let $f : X \to X$ be an orientation-preserving diffeomorphism with $f^\ast \fraks \cong \fraks$.
Let $W$ be an oriented smooth compact 4-manifold with boundary.
Suppose that:
\begin{enumerate}
\item[(i)] $f$ localizes to $W$ along an orientation-preserving smooth embedding $W \hookrightarrow X$.
\item[(ii)] $(c_1(\fraks|_{W})^2 -\sigma(W))/4 - (\chi(\mathring{W}) +\sigma(W))/2 > \max\Set{\deg{y} | y \in \HMhat_\ast(Y)\setminus\{0\}}$.
\end{enumerate}
Then $\FSW(X,\fraks,f)=0$.
\end{lem}

\begin{proof}
Set $Y=\del W$.
By the assumption (i), we get a relative diffeomorphism of $W$ that is isotopic to $f$.
We denote also by $f : W \to W$ this diffeomorphism by abuse of notation.
Set 
\[
x = \HMhat_\ast(\mathring{W}, \fraks|_{W},f)(\hat{1}) \in \HMhat_\ast(Y).
\]
By the families gluing (\cref{thm: Lin's gluing}), it suffices to prove $x=0$.
In the usual grading convention in monopole Floer homology, $\hat{1} \in \HMhat_\ast(S^3) = \Z[U]_{(-1)}$ is a generator of $\Z[U]$-module, where $\deg U=-2$. 
So $\deg(\hat{1})=-1$.
Because of the degree-shift formula \eqref{eq: deg of cob map} and assumption (ii), we have
\begin{align*}
\deg(x)
& = \frac{c_1(\fraks|_{W})^2 -\sigma(W)}{4} - \frac{\chi(\mathring{W}) +\sigma(W)}{2}\\
& > \max\Set{\deg{y} | y \in \HMhat_\ast(Y)\setminus\{0\}}.    
\end{align*}
Thus we have $x=0$.
\end{proof}

Now we recast the obstruction from Lemma~\ref{thm: general vinishing FSW} to $W$ with almost rationally plumbed homology spheres (introduced by N{\'e}methi \cite{nemethi2005ozsvath}) as boundaries. Recall that there is a canonical way to orient all almost-rational plumbed homology spheres by the orientations induced from the negative-definite plumbings to the boundaries. We refer to almost-rational plumbed homology spheres (in particular this includes all Seifert fibered spaces) with this orientation as {\it canonically oriented AR homology sphere}. For example, in this convention $S^3_{+1}(T_{2,3})=\Sigma(2,3,7)$ is considered to have positive orientation.
Let $d(Y)$ denote the Heegaard Floer correction term of an integral homology 3-sphere $Y$.

\begin{cor}
\label{thm: general vinishing FSW Seifert}
Let $X$ be a closed smooth oriented 4-manifold with $b^+(X)\geq3$, and $\fraks$ be a spin$^c$ structure on $X$ with $d(\fraks)=-1$.
Let $f : X \to X$ be an orientation-preserving diffeomorphism with $f^\ast \fraks \cong \fraks$.
Let $W$ be an oriented smooth compact 4-manifold.
Suppose that:
\begin{enumerate}
\item[(i)] $f$ localizes to $W$ along an orientation-preserving smooth embedding $W \hookrightarrow X$.
\item[(ii)] $(c_1(\fraks|_{W})^2 -\sigma(W))/4 - (\chi(\mathring{W}) +\sigma(W))/2 \geq d(\del W)$.
\item [(iii)] The oriented boundary $\del W$ is a canonically oriented AR homology sphere.
\end{enumerate}
Then $\FSW(X,\fraks,f)=0$.
\end{cor}

\begin{proof}
It follows from \cite{OS03plumbed} that the maximal grading of elements in $H \! F^{-}_\ast(Y)$ is given by the grading of the generator of the $U$-tower in $H \! F^{-}_\ast(Y)$.
Via the isomorphism of Heegaard Floer homology and monopole Floer homology, established in a series of papers started from \cite{KutluhanLeeTaubes20,ColinGhigginiHonda24,Taubes10embedded}, we have that $\HMhat_\ast(Y)$ is concentrated in degree $\leq d(Y)-1$.
This combined with \cref{thm: general vinishing FSW} completes the proof.    
\end{proof}

We will now use the obstructions above to produce several non-localizing results, concerning realizing exotic diffeomorphisms as Dehn twists and universal families corks. We begin with an immediate consequence:

\begin{thm}
\label{thm: non-localization to fixed orientation}
Let $X$ be a closed smooth oriented 4-manifold with $b^+(X)\geq3$ and let $f : X \to X$ be an exotic diffeomorphism with non-zero families Seiberg--Witten invariant for some spin$^c$ structure.
Then there exists no oriented compact integral homology 4-disk $W$ with the following property:
\begin{itemize}
\item[(i)] There is an orientation-preserving smooth embedding $W \hookrightarrow X$ along which $f$ localizes to $W$.
\item[(ii)] $\del W$ is a canonically oriented Seifert fibered 3-manifold.
\end{itemize}
\end{thm}

\begin{proof}[Proof of \cref{thm: non-localization to fixed orientation}]
Let $\fraks$ be a spin$^c$ structure on $X$ with $\FSW(X,\fraks,f) \neq 0$.
Suppose that there exists $W$ that satisfies the conditions (i), and (ii) in the statement.
Since $W$ is a homology 4-disk, it follows that $d(\del W)=0$.
Thus we have a contradiction from \cref{thm: general vinishing FSW Seifert}.
This completes the proof.
\end{proof}

In particular, many exotic diffeomorphisms of a closed 4-manifold are not obtained as Dehn twists on homology 4-disks with canonically oriented Seifert boundary.

\begin{ex}
It follows from \cref{subsection Infinite generation} that there are examples of $X$ for which we can get infinitely many exotic diffeomorphisms $\{f_i\}_{i=1}$ generating a $\Z^\infty$-summand in $\pi_0(\TDiff(X))$ for which each $f_i$ satisfies the same property as $f$ in \cref{thm: non-localization to fixed orientation}.

\end{ex}

We will now prove Theorem~\ref{thm: C is not universal}. In fact, we will prove a more general statement concerning almost-rational plumbed integer homology spheres. 
We prove the following:

\begin{thm}
\label{thm: Seiferts do not give universal families cork}
Let $W$ be an oriented smooth compact homology 4-disk.
Suppose that the oriented boundary $\del W$ is a canonically oriented AR plumbed homology sphere. Then both $W$ and $-W$ are not diff-universal.
\end{thm}

\begin{proof}[Proof of \cref{thm: Seiferts do not give universal families cork}]
Without loss of generality, we can assume that $\del W$ is oriented as negative-definite plumbing.
Since $W$ is a homology 4-disk, we have $d(W)=0$.
Let $X$ be a simply-connected closed oriented smooth 4-manifold that admits an exotic diffeomorphism $f : X \to X$ with $\FSW(X,\fraks,f) \neq0 $ for some spin$^c$ structure $\fraks$ on $X$ (such as diffeomorphisms given in \cref{thm: infinite generation general}).

Now suppose that $W$ is diff-universal.
Then there is an orientation-preserving smooth embedding $i : W \hookrightarrow X$ for the above $X$ along which $f$ localizes to $W$.
However, this is a contradiction by \cref{thm: general vinishing FSW Seifert}.

Next, suppose that $-W$ is diff-universal.
Let $X'$ be a simply-connected closed oriented smooth 4-manifold that admits an exotic diffeomorphism $f' : X' \to X'$ with $\FSW(-X',\fraks',f') \neq 0$ for some spin$^c$ structure $\fraks'$ on $-X'$.
The universality of $-W$ implies that 
there is an orientation-preserving smooth embedding $i' : -W \hookrightarrow X'$ for the above $X'$ along which $f'$ localizes to $-W$.
Now, $i'$ gives an orientation-preserving embedding $i' : W \hookrightarrow -X'$, and we get a contradiction from \cref{thm: general vinishing FSW Seifert} combined with $\FSW(-X',\fraks',f') \neq0 $.
This completes the proof.    
\end{proof}

\begin{rmk}
It is easy to extend \cref{thm: Seiferts do not give universal families cork} to rational homology sphere boundary case,  under the modification that, we have a spin$^c$ structure on $W$ that restricts to $Y$ with vanishing $d$-invariant.     
\end{rmk}

The proof of Theorem~\ref{thm: C is not universal} now follows:

\begin{proof}[Proof of \cref{thm: C is not universal}]
This is a corollary of \cref{thm: Seiferts do not give universal families cork}.
\end{proof}

It is also easy to see that our construction applies to homology disks whose boundary is not an AR plumbed integer homology sphere. For the discussion below, we refer to $S^3_{+1}(6_1)$ as the boundary of the Stevedore cork. Note that it bounds a contractible manifold, since $6_1$ is a slice knot. The Stevedore knot can be regarded as one of the simplest corks, see \cite[Lemma 7.3]{dai2022corks}. We show the following:

\begin{thm}\label{thm:akbulut_not_families_universal}
The Akbulut cork and the Stevedore cork are not universal families corks.
\end{thm}

\begin{proof}
Let $Y$ be the boundary of either the Akbulut cork or the Stevedore cork. Then it follows from \cite[Section 7]{dai2022corks} that $\HMhat_\ast(Y)$ has no elements in degree $ 0$. Then an analysis similar to that in the proof of Theorem~\ref{thm: Seiferts do not give universal families cork} gives us the conclusion.
\end{proof}

\begin{proof}[Proof of \cref{thm: not dehn twist fixed orientation}]
The proof is similar to that of \cref{thm: Seiferts do not give universal families cork}.
First, let us consider the case that $\calO$ is the canonical orientation.
In this case, let $X$ be a simply-connected closed oriented smooth 4-manifold that admits exotic diffeomorphisms $\{f_i : X \to X\}_{i=1}^\infty$ with $\FSW(X,\fraks,f_i) \neq0 $ for some spin$^c$ structures $\fraks_i$ on $X$ such that $\{f_i\}$ generates a $\Z^\infty$-summand of $\pi_0(\TDiff(X))$ (such as diffeomorphisms given in \cref{thm: infinite generation general}).
If there is $W$ that satisfies the condition (i-a) and (i-b), then this is a contradiction by \cref{thm: general vinishing FSW Seifert}.

Next, let us consider the case that $\calO$ is opposite to the canonical orientation.
In this case, we take a simply-connected closed oriented smooth 4-manifold $X'$ that admits exotic diffeomorphisms $\{f_i' : X' \to X'\}_{i=1}^\infty$ with $\FSW(-X',\fraks_i',f_i') \neq0 $ for some spin$^c$ structures $\fraks_i'$ on $-X'$ such that $\{f_i'\}$ generates a $\Z^\infty$-summand of $\pi_0(\TDiff(X'))$.
If there is $W$ that satisfies the condition (i-a) and (i-b), then again this is a contradiction by \cref{thm: general vinishing FSW Seifert}.
This completes the proof.
\end{proof}

Finally, we end by giving a proof of Theorem~\ref{thm: simultaneous}:

\begin{proof}[Proof of Theorem~\ref{thm: simultaneous}]
We pick $X:=M_X \# S^2 \times S^2$ from Lemma~\ref{lem: geography}. First, we construct an exotic diffeomorphism $f_1$ on $X$. Take any $\alpha \in H_2(X, \mathbb{Z})$. We define $f_1$ to be an exotic diffeomorphism in $\mathcal{Z}(\alpha)$ (see the proof of Theorem~\ref{thm: intro stabilized 4-manifolds}). Note that by construction
\begin{align}\label{non_localization_Dehn_twist_1}
\FSW(X,f_1, \mathfrak{s}_1) \neq 0, \; \text{for a} \; \spinc\text{-structure} \; \mathfrak{s}_1.
\end{align}
We now move on to defining $f_2$. Again borrowing notation from Lemma~\ref{lem: geography}, we define an exotic diffeomorphism $g$ of $M_{-X} \# S^2 \times S^2$. Here $g$ can be any exotic diffeomorphism of $M_{-X} \# S^2 \times S^2$ such that
\begin{align}\label{non_localization_dehn_twist}
\FSW(M_{-X} \# S^2 \times S^2, g, \fraks) \neq 0, \;\; \text{for a} \; \spinc\text{-structure} \; \mathfrak{s}.
\end{align}
In particular, as before we take $g$ to be any of the exotic diffeomorphism from $\mathcal{Z}(\beta)$ for any non-zero homology class $\beta \in H_{2}(M_{-X} \# S^2 \times S^2)$.

Now consider the orientation reversing diffeomorphism $\phi$ given in Lemma~\ref{lem: geography},
\[
\phi: M_X \# S^2 \times S^2 \rightarrow M_{-X} \# S^2 \times S^2.
\]
We define
\[
f_2 := \phi^{-1} \circ g \circ \phi
\]
Note that by definition, $f_2$ is an exotic diffeomorphism of $X$. We now claim that the set $\{f_1, f_2 \}$ does not localize simultaneously to any homology disk $W \subset X$ with $\partial W$ is a Seifert fibered space. Towards contradiction suppose such a $W$ exists. We now consider two cases. Firstly, let us assume that $\partial W$ is canonically oriented. Then an  argument similar to that in proof of Theorem~\ref{thm: non-localization to fixed orientation} implies
\[
\FSW(X,f_1,\mathfrak{s}') = 0 \; \; \text{for any} \; \spinc\text{-structure} \; \mathfrak{s}'.
\]
This contradicts Equation~\eqref{non_localization_Dehn_twist_1} which shows that $f_1$ does not localize to $W$. 

Now suppose that $\partial W$ has the non-canonical orientation. We now claim that $f_2$ does not localize to $W$. Indeed, if it did, then $g$ will localize to $\phi(W) \subset M_{-X} \# S^2 \times S^2$. However, since $\phi$ is orientation reversing, $\phi (\partial W)$ is a Seifert fibered space with canonical orientation. Again, this implies 
\[
\FSW(M_{-X} \# S^2 \times S^2, g, \fraks) = 0, \; \text{for any} \; \spinc\text{-structure} \; \mathfrak{s}.
\]
which contradicts Equation~\eqref{non_localization_dehn_twist}. This completes the proof.
\end{proof}
We end by proving Corollary~\ref{thm:non_dehn_twist}:
\begin{proof}[Proof of Corollary~\ref{thm:non_dehn_twist}]
We consider the 4-manifold $X$ and the pair of exotic diffeomorphism $f_1$ and $f_2$ constructed in the proof of Theorem~\ref{thm: simultaneous}. It follows from an argument similar to \cite[Theorem C]{auckly2015stable} that $f_1$ and $f_2$ are isotopic to identity in $X \# S^2 \times S^2$. Hence \cite[Theorem 4.2]{krushkal2024corks}, implies that there exists a contractible manifold $C \hookrightarrow X$ such that both $f_1$ and $f_2$ localize to $C$. Now if $\partial C$ is Seifert-fibered and either of $f_1$ or $f_2$ is isotopic to a Dehn twist along it, then it will contradict Theorem~\ref{thm: simultaneous}.
\end{proof}

\bibliographystyle{plain}
\bibliography{mainref}
\end{document}